\documentclass[12pt, reqno]{amsart}
\usepackage[top=1in, bottom=1in, left=1in, right=1in]{geometry}
\usepackage{amsfonts}

\usepackage{amsmath}
\usepackage{amssymb}
\usepackage{amsthm}
\usepackage{hyperref} 
\usepackage{empheq}
\usepackage{mathtools}
\usepackage{bbm}
\usepackage{manfnt}
\newtheorem{theorem}{Theorem}[section]

\newtheorem{corollary}[theorem]{Corollary}
\newtheorem{lemma}[theorem]{Lemma}
\newtheorem{proposition}[theorem]{Proposition}
\newtheorem{remark}[theorem]{Remark}
\newtheorem{conjecture}[theorem]{Conjecture}

\usepackage{xcolor}
\usepackage{comment}

\newcommand{\Q}{\mathbb{Q}}

\newcommand{\Z}{\mathbb{Z}}

\newcommand{\N}{\mathbb{N}}

\def\fa{\mathfrak{a}}

\def\cF{\mathcal{F}}

\def\cO{\mathcal{O}}

\newcommand{\cubic}[2]{\left( \frac{#1}{#2} \right)_3}

\usepackage[
backend=biber,
style=alphabetic,
maxbibnames=5,
sorting=nyt,
doi=false,isbn=false,url=false,
giveninits=true
]{biblatex}

\renewbibmacro{in:}{}

\DeclareFieldFormat*{title}{#1}

\newcommand{\psum}{\sideset{}{^*}\sum}

\title{Hecke $L$-functions Away From The Central Line} 
\author{Mohammad H. Hamdar}
\address{Department of Mathematics and Statistics, Concordia University, 1455 de Maisonneuve West, Montreal, Quebec, Canada H3G 1M8}
\email{mohammadhussein.hamdar@mail.concordia.ca}

\begin{document}
\maketitle
\begin{abstract}
 We compute the first moment of cubic Hecke $L$-functions over $\mathbb{Q}(\sqrt{-3})$ evaluated at any $s$ inside the critical strip. The first moment for $s<\frac{1}{2}$ is particularly interesting, and we show there is a phase transition at $s=\frac{1}{3}$. This extends the analogue result of David-Meisner \cite{DavidMeisner} for the first moment over function fields. As in their work, the computation of the moment at $s=\frac{1}{3}$ relies on a cancellation between two terms which are a priori not related: a main term of the principal sum which comes from cubes, and the contribution from infinitely many residues of Dirichlet series of cubic Gauss sums to the dual sum. The cancellation also improves the error term and exhibits a secondary term for all $s$. In particular, at $s=\frac{1}{2}$, we prove the existence of a secondary term of size $Q^{5/6}$, where the size of the family is $Q$.

  We conjecture that a similar behaviour would hold for higher order Hecke $L$-functions attached to $\ell^{th}$ order residue symbols, refining a function field conjecture of David and Meisner.

  The proof follows the steps of writing $L(s,\chi)$ as two finite sums with the approximate functional equation. Two main ingredients are then exploited: the bound on the second moment of $L(1/2+it,\chi)$ that follows from Heath-Brown's cubic large sieve, and the deep work of Kubota and Patterson which connects the Dirichlet series of cubic Gauss sums to metaplectic forms, and gives a formula for its residues in terms of the Fourier coefficients of metaplectic theta functions. This is only known for cubic Gauss sums, and not for general Gauss sums of order $\ell\geq 4$.
\end{abstract}

\section{Introduction and Main Results}

 Values of $L$-functions at the central critical point $s=\frac12$ are of fundamental importance in number theory. General conjectures for their moments in families are due to \cite{KS1,KS2, CFKRS}, but for any given family of $L$-functions, only the first few moments are known. In this paper, we are interested in the first moment, but in general for any $s \in [0,1]$.

Let $\cF_\ell$ be a family of primitive characters $\chi$ of fixed order $\ell$ over a number field $K$, and denote by $c_\chi$ the conductor of $\chi$. We say that $\cF_\ell$ has size $Q:=Q(X)$ if the quantity 
\begin{align*}
N_\ell (X) =\# \{ \chi \in \cF_\ell \;:\; N_{K/\mathbb{Q}}(c_\chi) \leq X \} 
\end{align*}
satisfies $N_\ell(X) \sim Q(X)$ as $X \rightarrow \infty$. Keating and Snaith \cite{KS1,KS2} conjectured that  if $\cF_\ell$ has size $Q$ then
\begin{equation}\label{KeatingSnaith}
\sum_{\substack{\chi \in \cF_{\ell}\\N_{K/\Q}(c_\chi)\leq X}} L(\tfrac12, \chi) \sim \begin{cases}  
g_2a_2Q\log Q & \ell = 2 \\
g_{\ell}a_{\ell}Q &\ell \geq 3,
\end{cases}
\end{equation}
where $g_{\ell}$ is the random matrix theory factor which depends only on the symmetry type of the family, and $a_{\ell}$ is the arithmetic factor which depends on the arithmetic of the family. 

 The quadratic case $\ell=2$ of \eqref{KeatingSnaith} was first proved by Jutila \cite{Jut} for characters over $\mathbb{Q}$, and by Andrade and Keating \cite{AndKeat} for characters over function fields. The cubic case $\ell=3$ was obtained by Luo \cite{Luo2004} for Hecke $L$-functions over $\mathbb{Q}(\sqrt{-3})$ (the Kummer case) where he considered a thin family of characters, and by Baier and Young \cite{BY} for Dirichlet $L$-functions over $\mathbb{Q}$ (the non-Kummer case). The analogues of the latter two results over function fields were recently proved by David, Florea and Lalin  \cite{DFL22}, where they now work with the full family of characters. For the quartic case, Gao and Zhao proved this conjecture for Hecke $L$-functions over $\mathbb{Q}(i)$ \cite{GaoZhao20} and conditionally under the Lindel\"{o}f hypothesis for Dirichlet $L$-functions over $\mathbb{Q}$ \cite{GaoZhao21}.

 From a random matrix theory point of view, the behavior in \eqref{KeatingSnaith} can be explained by the fact that quadratic characters have symplectic symmetries, while characters of order $\ell \geq 3$ have unitary symmetries. Therefore, the moments at $s=\frac12$ distinguish between quadratic characters and characters of order $\ell \geq 3$. In an attempt to distinguish characters of order $\ell\geq 3$ between themselves, Meisner \cite{Meisner} studied certain statistics over function fields $\mathbb{F}_q[t]$ which are different for every $\ell$. His results hold at the $q$-limit, but in a subsequent paper, David and Meisner \cite{DavidMeisner} formulated a conjecture for the average of the special value $L(\frac{1}{\ell},\chi)$ when $\chi$ varies over a family of characters of order $\ell$ over $\mathbb{F}_q[t]$ for fixed $q$. In the same paper, they prove the $\ell=3$ case of their conjecture over $\mathbb{F}_q[t]$ with fixed $q\equiv 1\pmod 3$.
 
 In this paper, we first state the analogue of the David-Meisner conjecture over number fields, and then prove the cubic case $\ell=3$. To state the conjecture, take a number field $K$ that contains the $\ell$-th roots of unity, and let $\cO_K$ be its ring of integers. For an ideal $\fa \in \cO_K$, one can define the $\ell$-th residue symbol $\chi_\fa$ (explicitly given in \cite{FF} for $\ell=2$ over function fields, but it is similar for any $\ell$ and over number fields, see \cite{FHL} and \cite{BGL}), which is a Hecke character of order $\ell$ defined over $\cO_K$. The Hecke $L$-function associated to $\chi_\fa$ is then defined by
\[L(s,\chi_\fa):=\sum_{0\neq \mathfrak{m}\subset\cO_K}\frac{\chi_\fa(\mathfrak{m})}{N(\mathfrak{m})^s} \ \ \text{for}\ \ \Re(s)>1, \]
where the sum is over nonzero integral ideals of $\cO_K$ and $N(\mathfrak{m})$ is the field norm of $\mathfrak{m}$. Let $\cF_\ell$ be a family of such characters of size $Q=Q(X)\geq 2$, and for each character $\chi$, let $c_\chi$ be its conductor. The conjecture of David-Meisner would then translate here to 
 \begin{conjecture}\label{conjecture} (David-Meisner) Let $w:(0,\infty)\to \mathbb{R}$ be a smooth function with compact support in $(0,1)$. For $\ell\geq 3$ and $s\in (0,1)$, we have that
\begin{align*}
\sum_{\chi \in \cF_\ell}  L(s, \chi) w \left( \frac{N (c_\chi)}{X} \right) \sim \begin{cases} C_{\ell,s, w} Q & \text{if}\ \ \frac{1}{\ell} < s<1 \\
C_{\ell,s, w} Q \log{Q} & \text{if}\ \ s = \frac{1}{\ell} \end{cases}
\end{align*}
where $C_{\ell,s, w}$  is a constant that depends on $\ell$, $s$ and the function $w$.
\end{conjecture}
 

We prove this for the case $K=\mathbb{Q}(\omega)$ where $\omega=e^{2\pi i/3}$, and for the family 
$$\cF_3= \left\{ \chi_q:=\cubic{\cdot}{q} \;:\; q \in \Z[\omega], \ q\equiv 1\pmod{9} \; \text{and} \; \mu_{\omega}^2(q) = 1 \right\}$$
 of all primitive cubic characters of squarefree (SF) conductor $q\in \mathbb{Z}(\omega)$, where $q\equiv 1\pmod{9}$ (see section 2.1 for the definition of the cubic residue symbol and for more details on the family, and note that $\cF_3$ has size $Q=X$ as $N_{\cF_3}(X)\sim X$). More precisely, we have the following result.
\begin{theorem}\label{Cubic}
Let $\epsilon>0$ and $w:(0,\infty)\to \mathbb{R}$ be a smooth function with compact support in $(0,1),$ with $\widetilde{w}$ its Mellin transform defined in \eqref{Mellin w}.

 For $0\leq s\leq 1$ with $s\neq \frac{1}{3}$, we have

\[\sum_{\chi_q \in \cF_3}L(s,\chi_q)w\left(\frac{N(q)}{Q}\right)=C_s\widetilde{w}(1)Q+E'_s\widetilde{w}\left(4/3-s\right)Q^{4/3-s}+O\left(Q^{\frac{11-6s}{10}+\epsilon}\right)\]
for explicit constants $C_s$ and $E'_s$ given in \eqref{Cubic C_z} and \eqref{E'_z} respectively.

 While at $s=1/3$, we have

\[\sum_{\chi_q \in \cF_3}L(1/3,\chi_q)w\left(\frac{N(q)}{Q}\right)=C_{1/3}\widetilde{w}(1)Q\log Q+D_{1/3}Q+O\left(Q^{9/10+\epsilon}\right)\]
for explicit constants $C_{1/3}$ and $D_{1/3}$ given in \eqref{Cubic C_1/3} and \eqref{Cubic D_1/3} respectively.
\end{theorem}

We note that Theorem \ref{Cubic} remains true if we consider any $s$ in the critical strip and not just real $s$. Thus, a phase transition at $\Re (s)=\frac13$ occurs on any horizontal line inside the critical strip. When $0\leq s<1/3$, the main and secondary terms in the asymptotic switch. We call this the \emph{``switching behaviour"} which we also observe in the quadratic case in Theorem \ref{Quadratic} below. Furthermore, for $s\in \left[0,\frac{1}{6}\right]\cup\left[\frac{7}{12},1\right]$, the error term dominates the respective secondary term.

\begin{remark}
At $s=\frac{1}{2}$, Theorem \ref{Cubic} gives
\[\sum_{\chi_q \in \cF_3}L(1/2,\chi_q)w\left(\frac{N(q)}{Q}\right)=C_{1/2}\widetilde{w}(1)Q+E'_{1/2}\widetilde{w}\left(5/6\right)Q^{5/6}+O\left(Q^{4/5+\epsilon}\right), \]
improving the result of Luo \cite{Luo2004}, where we obtain the secondary term $Q^{5/6}$ and improve his error term of size $Q^{21/22+\epsilon}$.
\end{remark}

A secondary term of size $Q^{5/6}$ is also conjectured for the second moment over cubic characters by Diaconu \cite[Conjecture 4.5]{Diaconu}, but it is still an open question to show the existence of that term. In \cite{chantal2024}, the authors give an asymptotic for the second moment over the same family $\mathcal{F}_3$, with an error term of size $Q^{5/6+\epsilon}$, just missing the secondary term.


 Theorem \ref{Cubic} also gives a better asymptotic than its function field analogue in \cite{DavidMeisner}, as we were able to capture secondary terms and get better error terms. We were additionally able to see the behaviour below the $1/3$ point, as well as at $s=1$. 

To our knowledge, there are very few papers investigating moments for $\Re (s)<1/2$. In \cite{FHL}, Friedberg, Hoffstein, and Lieman compute the first moment of $L$-functions $L^*(s, \chi)$, where $\chi$ runs over families of characters of order $\ell\geq 3$, and $1/\ell < \Re(s) < 1$. Their $L$-function $L^*(s, \chi)$  is a version of $L(s, \chi)$ with some Euler factors modified (depending on the conductor of $\chi$), and also include some arithmetic automorphic factors (see \cite[Eqn 1.3]{FHL}). Those requirements are forced from the properties of the double Dirichlet series that they introduce in order to study their moments, with techniques very different from ours. For $1/\ell < \Re(s) < 1$, they get an asymptotic for the first moment (running over all characters in thin families, and not only the primitive ones) with principal and secondary terms as in Theorem \ref{Cubic}, but with an additional tertiary term and a very small error term of size $Q^{\epsilon}$ (\cite[Theorem 1.3]{FHL}). They do not address the transition at $\Re(s)=1/\ell$ but they remark after Theorem 1.3 that for the case $s=\frac{1}{\ell}$, ``two main terms of order $Q$ must be carefully combined".

Our work (and that of \cite{DavidMeisner}) depends strongly on a cancellation between two terms from the moment computation, as seen in Theorem \ref{Asymptotics}: at $s=1/3$, $\mathcal{M}_1(1/3)$ and $\mathcal{M}_2(1/3)$ combine to give the main term (since $B=Q/A$), and at $s\neq 1/3$, the cancellation between $\mathcal{M}_1(s)$ and $\mathcal{M}_2(s)$ is essential to see the secondary term. A similar cancellation was also observed in \cite{DFL22} where the authors compute the first moment for the \textit{full families} of cubic characters over function fields at $s=1/2$. But unlike the case of the thin family $\mathcal{F}_3$ of this paper, the error terms for the full families are not strong enough to use the cancellation and exhibit the secondary term.

The strategy of the proof is as follows. We use the approximate functional equation to write our moment as a principal sum and a dual sum. We then treat each sum separately with different methods. The main term of the principal sum comes from the cubes; and to bound the error term, we use a bound on the second moment of the Hecke $L$-function in Theorem \ref{Bound on 2nd moment} which follows from Heath-Brown's cubic large sieve \cite{HB2000}. The dual sum involves cubic Gauss sums (which come from the sign of the functional equation) and is related to Kubota's metaplectic Dirichlet series defined in Section 4. We use deep results of Kubota, Patterson, and Heath-Brown$-$Patterson on metaplectic forms to establish an asymptotic for the dual sum. The main term of the dual sum comes from summing all the residues of this series. At the end, after exact computations of the explicit constants that arise in the asymptotic, we observe a delicate cancellation of certain terms when combining the two sums which is crucial to proving Theorem \ref{Cubic}.

One could ask if our proof could generalize to higher order characters to obtain a proof of Conjecture \ref{conjecture} for any $\ell$. Working with higher order characters becomes very intricate as the number field may have class number $>1$ and infinitely many units, which makes Artin's reciprocity difficult to track. In fact, the local obstructions coming from the character might restrict us to work with a version of $L(s,\chi)$ where some Euler factors are removed. It is worth noting that for bounding the error term in the principal sum, there is an analogue of the cubic large sieve for the higher order characters which was established by Blomer, Goldmakher, and Louvel \cite{BGL}. But the main obstacle is the dual sum computation which relies on an exact computation of the residues of the Dirichlet series of Gauss sums which are still shrouded in mystery for the $\ell$-fold cover of the metaplectic group when $\ell\geq 4$, and only very partial results are known. For the quartic case $\ell=4$, we mention the work of Suzuki \cite{Suzuki1} and later the conjecture of Eckhardt-Patterson \cite{EckhardtPatterson} that Suzuki then proved part of in \cite{Suzuki2}. These results have been recently leveraged to get bounds on quartic Gauss sums over primes, see \cite{daviddunnhamieh}. For general $\ell$, we refer the reader to the paper of Patterson \cite{PatGeneral}.

As a comparison with the cubic case, for $\ell=2$ (quadratic characters), we compute the first moment for any $s\in [0,1]$ in Theorem \ref{Quadratic} below, where a transition now happens at $s=1/2$. As far as we know, this was never written down anywhere in the literature. Let $\chi_q(.)=\left(\frac{.}{q}\right)_2$ be the quadratic character given by the Kronecker symbol and consider the family
\[\cF_2=\{\chi_q\;:\;q\in\mathbb{Z}_{>0},\; q\equiv 1\pmod 4,\; \text{and}\; \mu^2(q)=1 \} \]
which has size $Q=X$. We have the following result for the associated family of quadratic $L$-functions.

 \begin{theorem}\label{Quadratic} Let $\epsilon>0$ and $w:(0,\infty)\to \mathbb{R}$ be a smooth function with compact support in $(0,1)$.
 For any $0\leq s\leq 1$ with $s\neq \frac{1}{2}$, we have 

\[\sum_{\chi_q\in\cF_2}L(s,\chi_q)w\left(\frac{q}{Q}\right)=C_s\widetilde{w}(1)Q+C'_s\widetilde{w}(3/2-s)Q^{3/2-s}+O\left(Q^{1-\frac{s}{2}+\epsilon}\right),\]
for explicit constants $C_s$ and $C_s'$ given in \eqref{Quad C_z} and \eqref{Quad C_z'} respectively.

 While at $s=1/2$, we have

\[\sum_{\chi_q\in\cF_2}L(1/2,\chi_q)w\left(\frac{q}{Q}\right)=C\widetilde{w}(1)Q\log Q+DQ+O\left(Q^{3/4+\epsilon}\right)\]
for explicit constants $C$ and $D$ given in \eqref{Quad C} and \eqref{Quad D} respectively.
 \end{theorem}

Note that when $0\leq s<1/2$, the main and secondary terms in the asymptotic switch, again obeying the \emph{switching behaviour} which we observed in the cubic case; and at $s=0,1$, the error term dominates the respective secondary term. At $s=1/2$, we recover the result of Jutila \cite{Jut} with the same error term (the same result was also obtained in \cite{VinTakh}, and the error term was later improved in \cite{GoldfeldHoffstein} and \cite{Young}). It is worth mentioning that, similar to the cubic case, we also observe an explicit cancellation here when combining the principal and dual sums if $s\neq \frac{1}{2}$ (see Theorem \ref{Asymptotics0}).

Finally, Theorems \ref{Cubic} and \ref{Quadratic} and their proofs, make it reasonable to refine the David-Meisner conjecture by including secondary terms (as in \cite{FHL} for $\Re(s)>1/\ell$) and also predicting the \emph{switching behaviour} when $\Re(s)\in \left[0,\frac{1}{\ell}\right)$. Indeed, recalling that $\cF_{\ell}$ is the size $Q=Q(X)$ family of characters of order $\ell$, we conjecture the following.

\begin{conjecture}
For $\ell\geq 4$ and $s\in \mathbb{C}$, we have that
\begin{align*}
\sum_{\chi \in \mathcal{F}_{\ell}} L(s, \chi)w\left( \frac{N(c_{\chi})}{X}\right) = \begin{cases}  
C_{\ell,s,w}Q+D_{\ell,s,w}Q^{1+\frac{1}{\ell}-s}\left(1+o(1)\right) & \ \frac{1}{\ell}<\Re(s)\leq 1 \\
C_{\ell,w}Q\log Q+D_{\ell,w}Q\left(1+o(1)\right) & \ \Re(s)=\frac{1}{\ell}\\
D_{\ell,s,w}Q^{1+\frac{1}{\ell}-s}+C_{\ell,s,w}Q\left(1+o(1)\right) & \ 0\leq \Re(s)<\frac{1}{\ell}.
\end{cases}
\end{align*}
\end{conjecture}

In the proofs (right before Theorem 2.6) we will start writing $L(z,\chi_q)$, and leave $s$ as a variable of integration.

\subsection*{Acknowledgments} I am grateful to Chantal David for introducing this project to me and for her guidance and enthusiasm throughout. I would like to thank Cihan Sabuncu for many fruitful discussions and beneficial suggestions, and Alexander Dunn for his relevant comments. I also thank Andrew Granville, Alexandre de Faveri, Sun-Kai Leung, and Patrick Meisner for enlightening conversations related to this problem.

\section{Background and Preliminaries}

\subsection{Arithmetic in the Eisenstein Ring}

  
Recall that $\omega=-\frac{1}{2}+i\frac{\sqrt{3}}{2}$ and $\mathbb{Q}(\omega)$ is the Eisenstein quadratic field with ring of integers $\mathbb{Z}[\omega]$, class number $1$, discriminant $-3$, and a finite group of units $\{\pm 1,\pm \omega, \pm \omega^2\}=\langle -\omega\rangle$. Denote by $\lambda:=1-\omega$ the unique ramified prime in $\mathbb{Z}[\omega]$. Any non-zero element in $\mathbb{Z}[\omega]$ can be uniquely written as $\zeta \lambda^k m$ where $\zeta\in\langle -\omega\rangle, k\in \mathbb{Z}_{\geq 0},$ and $m\in \mathbb{Z}[\omega]$ with $m\equiv 1\pmod 3$.  For any $n\in \mathbb{Z}[\omega]$, let $\mu(n):=\mu_{\mathbb{Z}[\omega]}(n)$ denote the standard M\"{o}bius function in $\mathbb{Q}(\omega)$, and let 
$$\sigma(n):=\sum_{\substack{d\equiv 1(\mathrm{mod}\ 3)\\d|n}}N(d)$$
denote the sum of primary divisors function.

 For $a,\pi\in \mathbb{Z}[\omega]$ and $\pi$ a prime coprime to $\lambda$, the cubic residue symbol mod $\pi$ is defined by
\[\left(\frac{a}{\pi}\right)_3\equiv a^{\frac{N(\pi)-1}{3}} \pmod \pi\]
taking values in $\{1,\omega,\omega^2 \}$ if $a$ and $\pi$ are coprime, and equals $0$ if $\pi|a$. This symbol is multiplicative in $a$ and can also be extended multiplicatively for all $b\in \mathbb{Z}[\omega]$ by setting
\[\left(\frac{a}{b}\right)_3=\prod_{i}\left(\frac{a}{\pi_i}\right)_3^{e_i} \]
where $b=\prod_i \pi_i^{e_i}$ for primes $\pi_i$. Moreover, the cubic symbol obeys cubic reciprocity: If $a,b\equiv 1\pmod 3$, then
\[\left(\frac{a}{b}\right)_3=\left(\frac{b}{a}\right)_3. \]
We also have supplementary laws of cubic reciprocity concerning units and the ramified prime $\lambda$. Given
\[c\equiv 1+\alpha_2\lambda^2+\alpha_3\lambda^3 \pmod 9\ \ \text{with} \ \ \alpha_2,\alpha_3\in\{-1,0,1 \}, \]
then
\begin{equation}\label{supplemet}
\left(\frac{\omega}{c}\right)_3=\omega^{\alpha_2}\ \ \text{and}\ \ \left(\frac{\lambda}{c}\right)_3=\omega^{-\alpha_3}.  
\end{equation}
 For $q\in \mathbb{Z}[\omega]$ with $q\equiv 1\pmod 3$, the Dirichlet character
 \begin{equation}\label{chi_q}
 \chi_q(m):=\left(\frac{m}{q}\right)_3, \ \ m\in\mathbb{Z}[\omega], 
 \end{equation}
  on $\mathbb{Z}[\omega]/q\mathbb{Z}[\omega]$ is a well defined Hecke character when $q\equiv 1\pmod 9$, since then $\chi_q(\omega)=1$ by \eqref{supplemet}. 
  In this case, $\chi_q$ can now be regarded as a ray class character for the ray class group $h_{(q)}=I_{(q)}/P_{(q)}$, where $I_{(q)}=\{\mathfrak{m}\subset\mathcal{I}: (\mathfrak{m},(q))=1\}$ and $P_{(q)}=\{(m)\subset\mathcal{P}: m\equiv 1\pmod q\}$ with $\mathcal{I}$ and $\mathcal{P}$ denoting the group of fractional ideals in $\mathbb{Q}(\omega)$ and the subgroup of principal ideals respectively.

The character $\chi_q$ is primitive when $q$ is cubefree, i.e. $q=q_1q_2^2$ where $q_1,q_2\in \mathbb{Z}[\omega]$, $q_1, q_1\equiv 1\pmod 3,$ and $\mu^2(q_1q_2)=1$. Each $\chi_q$ in this case has conductor $q_1q_2\mathbb{Z}[\omega]$. Taking further $q$ to be SF (i.e. $q_2=1$), we get that the conductor of $\chi_q$ is now $q\mathbb{Z}[\omega]$, and so we define the family
\begin{equation}\label{family F_3}
\cF_3= \left\{\chi_q(.):=\left(\frac{.}{q}\right)_3 \;:\; q \in \Z[\omega], \ q\equiv 1\pmod{9} \; \text{and} \; \mu^2(q) = 1 \right\}.
\end{equation}

\subsection{Cubic Gauss Sums}

For $z\in \mathbb{C}$, denote by
\[e\left(z\right):=e^{2\pi i\text{Tr}(z)}=e^{2\pi i(z+\bar{z})}, \]
where $\text{Tr}(z)$ is the Trace from $\mathbb{Q}(\omega)$ to $\mathbb{Q}$. If $c\in \mathbb{Z}[\omega]$ and $c\equiv 1\pmod 3$, the cubic Gauss sum is defined by
\begin{align}\label{Gauss sum}
g(c)=\sum_{d(\mathrm{mod}\ c)}\left(\frac{d}{c}\right)_3e\left(\frac{d}{c}\right).
\end{align}
We have an exact formula for the cube of $g(c)$ \cite[pp. 443, 445]{H}:
\[g(c)^3=\mu(c)c^2\bar{c},\]
which implies that $g(c)$ is only supported on squarefree moduli $c$. In that case, the size of the Gauss sum is
\[|g(c)|=|c|=N(c)^{1/2},\]
 and the normalised Gauss sum is thus given by
 \[\tilde{g}(c):=\frac{g(c)}{N(c)^{1/2}}. \]
 For $\nu\in\mathbb{Z}[\omega]$, we also consider the shifted (or generalized) Gauss sum
\[g(\nu,c):=\sum_{d(\mathrm{mod}\ c)}\left(\frac{d}{c}\right)_3e\left(\frac{\nu d}{c}\right), \]
 so that $g(c)=g(1,c)$. The next two lemmas record standard algebraic properties of these sums.
 
\begin{lemma}\label{properties of g}
Let $n,n_1,n_2,m,r\in \mathbb{Z}[\omega]$ with $n,n_1,n_2\equiv 1\pmod 3$. Then,
\\
If $(m,n)=1$, \begin{equation}\label{1}
g(mr,n)=\overline{\left(\frac{m}{n}\right)_3}g(r,n).
\end{equation}
If $(n_1,n_2)=1$, 
\begin{equation}\label{2}
 g(r,n_1n_2)=g(rn_1,n_2)g(r,n_1).
\end{equation}

\end{lemma}

Note that formulas \eqref{1} and \eqref{2} in Lemma \ref{properties of g} give an important property of $g(c)$, which is twisted multiplicativity. Namely, for $c=ab$ where $a,b \in \mathbb{Z}[\omega]$ and $(a,b)=1$, we have that \cite[pp. 443, 445]{H}
 \[g(ab)=\overline{\left(\frac{a}{b}\right)} g(a)g(b).\] 
 The next lemma explains the local behavior of the Gauss sums at prime arguments.
\begin{lemma}\label{local properties of g}
Let $\pi, r\in\mathbb{Z}[\omega]$ with $\pi,r\equiv 1\pmod 3$, $\pi$ a prime, and $(\pi,r)=1$. Let $k,j$ be integers with $k>0$ and $j\geq 0$.

If $k\neq j+1$, then 
\[ g(r\pi^j,\pi^k)=\begin{cases}
            
			\varphi(\pi^k) &\text{if $3\mid k,\ k\leq j$}\\
            0 & \text{otherwise}.
		 \end{cases}\]
		 
If $k=j+1$, then 
\[ g(r\pi^j,\pi^k)=N(\pi^j)\times\begin{cases}
            -1 & \text{if $3\mid k$}\\
			g(r,\pi) & \text{if $k\equiv 1\pmod 3$}\\
            \overline{g(r,\pi)} & \text{if $k\equiv 2\pmod 3$}.
		 \end{cases}\]

\end{lemma}

 \subsection{Hecke $L$-functions over $\mathbb{Q}(\omega)$}
 
 Let $\chi_q\in \mathcal{F}_3$ defined in \eqref{family F_3}. Its Hecke $L$-function defined as
 \[L(s,\chi_q)=\sum_{0\neq\mathfrak{m}\subset\mathbb{Z}[\omega]}\frac{\chi_q(\mathfrak{m})}{N(\mathfrak{m})^s}, \]
  converges absolutely for $\Re(s)>1$ and admits the Euler product
 \[L(s,\chi_q)=\prod_{\mathfrak{p}}\left(1-\frac{\chi_q(\mathfrak{p})}{N(\mathfrak{p})^s}\right)^{-1}, \]
 where the product runs over nonzero prime ideals of $\mathbb{Z}[\omega]$. Note that $\chi_q(\mathfrak{m})=0$ if $\mathfrak{m}$ is not coprime to $q$. Hecke showed that the completed $L$-function defined as
 \[\Lambda(s,\chi_q):=\left(|D_K|N(q)\right)^{s/2}(2\pi)^{-s}\Gamma(s)L(s,\chi_q), \] 
 (where $D_K=-3$ the discriminant of $\mathbb{Q}(\omega)$) satisfies the functional equation
 \[\Lambda(s,\chi_q)=\epsilon(\chi_q)\Lambda(1-s,\overline{\chi_q}), \]
 where $\epsilon(\chi_q)=\tilde{g}(q)$ is the root number. We note that $L(s,\chi_q)$ coincides with the $L$-function associated to the newform $f(z)\in S_k(\Gamma_0(N),\chi)$ given by
 \[f(z)=\sum_{\mathfrak{m}\subset\mathbb{Z}[\omega]}\chi_q(\mathfrak{m})e^{2\pi izN(\mathfrak{m})} \] where $$N=|D_K|N(q)\, \ \text{and}\, \ \chi(n)=\left(\frac{D_K}{n}\right)_2\chi_q(n),$$ see \cite[Theorem 12.5]{IwaniecTopicsBook}. 
 
 Inside the critical strip, analytically convenient expressions of $L(s,\chi_q)$ are given by the approximate functional equation:
 
 \begin{proposition}(\cite[Theorem 5.3]{IwK})\label{AFE}
 Let $G(u)$ be any function which is holomorphic and bounded in the strip $-4<\Re(u)<4$, even, and normalized by $G(0)=1$. Let $X>0$, then for $s=\sigma+it$ in the strip $0\leq \sigma\leq 1$ we have
 \begin{align*}
L(s,\chi_q)=\sum_{0\neq \mathfrak{m}\subset\mathbb{Z}[\omega]}\frac{\chi_q(\mathfrak{m})}{N(\mathfrak{m})^s}V_s\left(\frac{N(\mathfrak{m})}{X\sqrt{3N(q)}}\right)+\epsilon(s,\chi_q)\sum_{0\neq \mathfrak{m}\subset\mathbb{Z}[\omega]}\frac{\overline{\chi_q(\mathfrak{m})}}{N(\mathfrak{m})^{1-s}}V_{1-s}\left(\frac{XN(\mathfrak{m})}{\sqrt{3N(q)}}\right)
 \end{align*}
 where $V_s(y)$ is a smooth function defined by
 \[V_s(y)=\frac{1}{2\pi i}\int_{(3)}y^{-u}\frac{G(u)}{u}g_s(u)\,du \ \
 \text{with}\ \ g_s(u)=(2\pi)^{-u}\frac{\Gamma\left(s+u\right)}{\Gamma\left(s\right)}, \]
and
\begin{equation}\label{root number}
\epsilon(s,\chi_q)=(2\pi)^{2s-1}(3N(q))^{1/2-s}\frac{\Gamma\left(1-s\right)}{\Gamma\left(s\right)}\tilde{g}(q).
\end{equation}
 \end{proposition} 
 
 The smooth weight $V_s$ and its derivatives satisfy the following bounds which we will find useful throughout.
 
 \begin{lemma}(\cite[Proposition 5.4]{IwK})\label{Bounds on V_s}
 Let $a\in\mathbb{Z}_{\geq 0}$ and $\alpha>0$. For $\Re(s+1)\geq 3\alpha,$ the derivatives of $V_s(y)$ satisfy
 \[y^aV_s^{(a)}(y)\ll_{a,E}\left(1+\frac{y}{\sqrt{(|s|+3)(|s+1|+3)}}\right)^{-E} \]
 for any $E>0$, and
 \[ y^aV_s^{(a)}(y)=\delta_a+O_{a,\alpha}\left(\left(\frac{y}{\sqrt{(|s|+3)(|s+1|+3)}}\right)^{\alpha}\right)\]
 where $\delta_0=1$ and $\delta_a=0$ for $a\geq 1$.
 \end{lemma}
 
 \begin{remark}
 Note that in Proposition \ref{AFE} and Lemma \ref{Bounds on V_s} we are using the fact that the gamma factor $\gamma(s,\chi_q)$ of $L(s,\chi_q)$ is given by (see section 5.10 in \cite{IwK})
 \[\gamma(s,\chi_q)=\pi^{-s}\Gamma\left(\frac{s}{2}\right)\Gamma\left(\frac{s+1}{2}\right)=2\sqrt{\pi}(2\pi)^{-s}\Gamma(s), \]
 where the last equality follows from the duplication formula for $\Gamma(s)$.
 \end{remark}
 
 We will now start using the variable $z$ for the input of the $L$-functions, and leave $s$ as a variable of integration.
 
 Recall that each nonzero ideal $\mathfrak{m}\subset\mathbb{Z}[\omega]$ has a unique generator of the form $\lambda^km$ where $k\geq 0$ and $m\in\mathbb{Z}[\omega]$ with $m\equiv 1 \pmod 3$. Hence $\chi_q(\mathfrak{m})=\left(\frac{m}{q}\right)_3$, since $\left(\frac{\lambda}{q}\right)_3=1$ from \eqref{supplemet}. Thus, in view of Proposition \ref{AFE}, the first moment equals

\[\sum_{\substack{q\in \mathbb{Z}[\omega]\\ q\equiv 1(\mathrm{mod}\ 9)\\ q\, \text{SF}}}L(z,\chi_q)\omega\left(\frac{N(q)}{Q}\right)=\mathcal{M}_1(z)+\mathcal{M}_2(z),\]
where $\mathcal{M}_1(z)$ is the principal sum given by
\[ \mathcal{M}_1(z):=\sum_{\substack{q\in \mathbb{Z}[\omega]\\ q\equiv 1(\mathrm{mod}\ 9)\\ q\, \text{SF}}}\sum_{k\geq 0}\sum_{\substack{m\in \mathbb{Z}[\omega]\\ m\equiv 1(\mathrm{mod}\ 3)}}\frac{\left(\frac{m}{q}\right)_3}{3^{kz}N(m)^{z}}V_{z}\left(3^{k-\frac{1}{2}}\frac{N(m)}{A_q}\right)\omega\left(\frac{N(q)}{Q}\right),\]
and $\mathcal{M}_2(z)$ is the dual sum given by
\[\mathcal{M}_2(z):=\sum_{\substack{q\in \mathbb{Z}[\omega]\\  q\equiv 1(\mathrm{mod}\ 9)\\ q\, \text{SF}}}\epsilon(z,\chi_q)\sum_{k\geq 0}\sum_{\substack{m\in \mathbb{Z}[\omega]\\  m\equiv 1(\mathrm{mod}\ 3)}}\frac{\overline{\left(\frac{m}{q}\right)_3}}{3^{k(1-z)}N(m)^{1-z}}V_{1-z}\left(3^{k-\frac{1}{2}}\frac{N(m)}{B}\right)\omega\left(\frac{N(q)}{Q}\right).\]
Here, we chose the parameters of the approximate functional equation as in \cite{BY}. Take $A_q=\frac{AN(q)}{Q}$ and $B_q=B$, where $A$ and $B$ are defined as $AB=Q$, which gives $A_qB=N(q)$ as required. Note that in this case $A_q=\frac{N(q)}{B}=A\frac{N(q)}{Q}\asymp A$.

Finally, let $\widetilde{w}$ be the Mellin transform of $w$ defined by
\begin{equation}\label{Mellin w}
\widetilde{w}(z)=\int_0^{\infty}x^{z-1}w(x)\,dx.
\end{equation}

In what follows, we obtain asymptotics for both the principal and the dual sums. We summarize these in the following theorem.

\begin{theorem}\label{Asymptotics}
For any $\epsilon>0$ and $z\in [0,1]$ we have
\[\mathcal{M}_1(z)=\begin{cases}
C_z\widetilde{\omega}(1)Q+D_z\widetilde{\omega}(4/3-z)QA^{1/3-z}+O\left(QA^{-z}+Q^{1/2+\epsilon}A^{1-z+\epsilon}\right)& \text{if} \ \ z\neq 0,\frac{1}{3}\\
 D_0\widetilde{\omega}(4/3)QA^{1/3}+O(QA^{\epsilon}+Q^{1/2+\epsilon}A^{1+\epsilon}) & \text{if} \ z=0\\
 D_0\widetilde{\omega}(4/3)QA^{1/3}+C_0\widetilde{\omega}(1)Q+O(QA^{-1/6+\epsilon}+Q^{1/2+\epsilon}A^{1+\epsilon})& \text{if} \ z=0 +\text{GRH}\\
 C_{1/3}\widetilde{\omega}(1)Q\log A+C_2Q+O\left(QA^{-1/3}+Q^{1/2+\epsilon}A^{2/3+\epsilon}\right)& \text{if} \ z=\frac{1}{3}
\end{cases}
\]
where $C_z, D_z, C_{1/3}$ and $C_2$ are constants given explicitly in \eqref{Cubic C_z}, \eqref{Cubic D_z}, \eqref{Cubic C_1/3}, and \eqref{Cubic C_2} respectively, and by GRH we mean RH for the Dedekind zeta function $\zeta_{\mathbb{Q}(\omega)}$. Moreover,  
\[\mathcal{M}_2(z)=\begin{cases}-D_z\widetilde{w}\left(4/3-z\right)Q^{4/3-z}B^{z-1/3}+E'_z\widetilde{w}\left(4/3-z\right)Q^{4/3-z}\\
\hspace{1.5cm}+O\left(Q^{4/3-z}B^{z-2/3+\epsilon}+Q^{1-z+\epsilon}B^{1/4+z+\epsilon}\right)& \text{if} \ z\neq \frac{1}{3}\\
C_{1/3}\widetilde{\omega}(1)Q\log B+C_2'\widetilde{\omega}(1)Q+O\left(QB^{-1/3+\epsilon}+Q^{2/3+\epsilon}B^{7/12+\epsilon}\right)& \text{if} \ z=\frac{1}{3}

\end{cases}  \]
where $E'_z$ and $C'_2$ are constants given explicitly in \eqref{E'_z} and \eqref{Cubic C_2'} respectively.
\end{theorem}

As $B=Q/A$, the main terms of $\mathcal{M}_1(1/3)$ and $\mathcal{M}_2(1/3)$ combine to give $C_{1/3}Q\log Q$. For $z> 1/3,$ the secondary term of $\mathcal{M}_1(z)$ cancels with the main term of $\mathcal{M}_2(z)$. And for $z<1/3,$ the main term of $\mathcal{M}_1(z)$ cancels with the secondary term of $\mathcal{M}_2(z)$. Now choosing $A=Q^{3/5}$ and $B=Q^{2/5}$ will give the best error term for any $z$, and setting 
\begin{equation}\label{Cubic D_1/3}
D_{1/3}:=C_2+\widetilde{\omega}(1)C_2',
\end{equation}
 Theorem \ref{Cubic} will then follow immediately from Theorem \ref{Asymptotics}.

\section{The Principal Sum}  

 We want to evaluate the principal sum

\[\mathcal{M}_1(z)=\sum_{k\geq 0}3^{-zk}\sum_{\substack{m\in \mathbb{Z}[\omega]\\ m\equiv 1(\mathrm{mod}\ 3)}}\frac{1}{N(m)^{z}}\sum_{\substack{q\in \mathbb{Z}[\omega]\\ q\equiv 1(\mathrm{mod}\ 9)\\ q\, \text{SF}}}\left(\frac{m}{q}\right)_3V_{z}\left(\frac{3^k N(m)}{A}\frac{Q}{N(q)}\right)\omega\left(\frac{N(q)}{Q}\right).\]
To detect the squarefree condition, we use
\[\sum_{\ell^2\mid q}\mu(\ell)=
\begin{cases}
1 & q\, \text{is squarefree}\\
0 & \text{otherwise}
\end{cases}
 \]
 to get
 \begin{align*}
\mathcal{M}_1(z)&=\sum_{k\geq 0}3^{-zk}\sum_{\substack{\ell\in \mathbb{Z}[\omega]\\ \ell\equiv 1(\mathrm{mod}\ 3)}}\mu(\ell)\sum_{\substack{m\in \mathbb{Z}[\omega]\\ m\equiv 1(\mathrm{mod}\ 3)}}\frac{\left(\frac{m}{\ell^2}\right)_3}{N(m)^{z}}\mathcal{M}_1(z,\ell,m,k),
\end{align*}
where 
\[\mathcal{M}_1(z,\ell,m,k)=\sum_{\substack{q\in \mathbb{Z}[\omega]\\ q\equiv 1(\mathrm{mod}\ 9)}}\left(\frac{m}{q}\right)_3V_{z}\left(\frac{3^kN(m)}{A}\frac{Q}{N(q\ell^2)}\right)\omega\left(\frac{N(q\ell^2)}{Q}\right). \]
Using Mellin inversion, we can write

\[V_{z}\left(\frac{3^kN(m)}{A}\frac{Q}{N(q\ell^2)}\right)\omega\left(\frac{N(q\ell^2)}{Q}\right)=\frac{1}{2\pi i}\int_{(2)} \left(\frac{Q}{N(q\ell^2)} \right)^s\widetilde{f}_z(s,m,k)\,ds\]
where $$\widetilde{f}_z(s,m,k)=\int_{0}^{\infty}V_{z}\left(\frac{3^kN(m)}{Ax}\right)\omega(x)x^{s-1}\,dx.$$
Integrating by parts and using the bounds for $V_{z}$ and its derivatives in Lemma \ref{Bounds on V_s}, we see that $\widetilde{f}_z$ satisfies
 
\[\widetilde{f}_z(s,m,k)\ll(1+\lvert s\rvert)^{-E}(1+3^kN(m)/A)^{-E}, \]
for $\Re(s)\geq \frac{1}{4}$ and any $E>0$ and $0\leq z\leq 1$ (similar to section 3.1 in \cite{BY}). With this notation, 
\[ \mathcal{M}_1(z,\ell,m,k)=\frac{1}{2\pi i}\int_{(2)} \sum_{\substack{q\in \mathbb{Z}[\omega]\\ q\equiv 1(\mathrm{mod}\ 9)}}\left(\frac{m}{q}\right)_3\left(\frac{Q}{N(q\ell^2)} \right)^s\widetilde{f}_z(s,m,k)\,ds, \]
and using cubic reciprocity $\chi_q(m)=\chi_m(q)$ (since both $m,q\equiv1\mod 3$), we get that  
$$
\mathcal{M}_1(z,\ell,m,k)=\frac{1}{2\pi i}\int_{(2)}\left(\frac{Q}{N(\ell^2)} \right)^s\sum_{\substack{q\in \mathbb{Z}[\omega]\\ q\equiv 1(\mathrm{mod}\ 9)}}\frac{\chi_m(q)}{N(q)^s}\widetilde{f}_z(s,m,k)\,ds.
$$
To write the sum over $q$ as a Hecke $L$-function with character $\chi_m$, we use 
the following orthogonality relations that detect the congruence mod $9$ condition inside the congruence mod $3$ class. This can be found in \cite{chantal2024}, equation (7.32).
 
 \begin{lemma}\label{detect mod 9}
 Let $q,c\equiv 1\pmod 3$ and $\chi_c(.)=(\frac{.}{c})_3$ as above, then
 \[\mathbbm{1}_{q\equiv c\pmod{9}}(q)=\frac{1}{9}\sum_{a,b=0}^2 \chi_c(\omega^a\lambda^b)\overline{\chi_q(\omega^a\lambda^b)} =\frac{1}{18}\sum_{`\eta\mid 3\text{'}}\chi_c(\eta)\overline{\chi_q(\eta)}. \]
 where by `$\eta|3$' we mean that we are counting all divisors, not up to units.
 \end{lemma}
 Using Lemma \ref{detect mod 9} with $c=1$, we get
\begin{align*}
\mathcal{M}_1(z,\ell,m,k)=\frac{1}{9}\frac{1}{2\pi i}\sum_{a,b=0}^2\int_{(2)}\left(\frac{Q}{N(\ell^2)} \right)^sL(s,\psi_{a,b}\chi_m)\widetilde{f}_z(s,m,k)\,ds,
\end{align*}
where
 $$\psi_{a,b}(.):=\overline{\left(\frac{\omega^a\lambda^b}{.}\right)_3}$$
is a Hecke character on $\mathbb{Z}(\omega)$ of modulus $9$, in view of the reciprocity laws in \eqref{supplemet}.

 Now, the Hecke $L$-function $L(s,\psi_{a,b}\chi_m)$ 
 has a pole at $s=1$ only when $\psi_{a,b}$ is trivial and $m$ is a cube. We will estimate $\mathcal{M}_1(z)$ by moving the contour of integration of $\mathcal{M}_1(z,\ell,m,k)$ to the line $\Re(s)=1/2$. Denote by $\mathcal{M}_0(z)$ the contribution of $\mathcal{M}_1(z)$ of the poles at $s=1$, and by $\mathcal{M}'_0(z)$ the contribution of the half line, namely
\begin{equation}\label{Main Term}
\mathcal{M}_0(z)=\frac{Q}{9}\sum_{k\geq 0}3^{-kz}\sum_{\ell\equiv 1(\mathrm{mod}\ 3)}\frac{\mu(\ell)}{N(\ell^2)}\sum_{m\equiv 1(\mathrm{mod}\ 3)}\frac{\left(\frac{m^3}{\ell^2}\right)_3}{N(m^{3z})}\widetilde{f}_z(1,m^3,k)Res_{s=1}L(s,\chi_{m^3}),
\end{equation}
and 
\begin{align}\label{Error Term}
\mathcal{M'}_0(z)&\ll Q^{1/2}\sum_{k\geq 0}3^{-zk}\sum_{\substack{\ell\equiv 1(\mathrm{mod}\ 3)\\ N(\ell)\leq Q^{1/2}}}\frac{1}{N(\ell)}\sum_{\substack{m\equiv 1(\mathrm{mod}\ 3)\\ N(m)\leq 3^{-k} A^{1+\epsilon}}}\frac{1}{N(m)^{z}}\nonumber\\
&\times\sum_{\psi \ \mathrm{mod}\ 9} \int_{(1/2)}\left\lvert L(s,\psi\chi_m) \widetilde{f}_z(s,m,k)\right\rvert\,ds.
\end{align}

\subsection{Bounding the Error Term}

To obtain an upper bound on the error term, we need the following bound on the second moment of the $L$-function at $1/2+it$ for any $t\in \mathbb{R}$. 

\begin{theorem}\label{Bound on 2nd moment} For $X\geq 1$ and $t\in\mathbb{R}$ we have the bound
\[\sum_{N(m)\leq X}\lvert L(1/2+it,\chi_m)\rvert^2\ll_{\epsilon} X^{1+\epsilon}(1+\lvert t\rvert)^{\frac{1}{2}+\epsilon} \]
for any $\epsilon>0$.
\end{theorem}
This result can be found in \cite{BGL} for general $\ell^{th}$ order characters. We give a detailed proof for the cubic case here for completeness. The proof of Theorem \ref{Bound on 2nd moment} requires, as an essential part, the following mean value estimate for character sums in $\mathbb{Z}[\omega]$.
\begin{lemma}\label{cubic large sieve}
Let $M,N>0$ and $c_n$ be an arbitrary sequence of complex numbers where $n$ runs over $\mathbb{Z}[\omega]$. Then
\[\psum_{N(m)\leq M}\left\lvert\, \psum_{N(n)\leq N}c_n\left(\frac{n}{m}\right)_3\right\rvert^2\ll_{\epsilon} \left(M+N+(MN)^{2/3} \right)\left(MN \right)^{\epsilon}\psum_{N(n)\leq N}\lvert c_n \rvert^2,\]
for any $\epsilon>0$ and where the star over the sum indicates that it is over squarefree elements congruent to $1$ modulo $3$. 
\end{lemma}

Lemma \ref{cubic large sieve} is known as the cubic large sieve and was proved by Heath-Brown in \cite[Theorem 2]{HB2000}. 
 Comparing the cubic large sieve with its quadratic variant (see Lemma \ref{quadrativ large sieve} below), Heath-Brown speculated that the extra $(MN)^{2/3}$ factor from the cubic one can be possibly removed to get the better bound as in the quadratic case, which became the standard belief of most experts in the area. However, Dunn and Radziwiłł recently \cite{DR} showed, under GRH for Hecke $L$-functions over $\mathbb{Q}(\omega)$, in their work on proving Patterson's conjecture, that the $(MN)^{2/3}$ term in fact can't be removed and the cubic large sieve is thus optimal (up to factors of the form $X^{\epsilon}$). It is worth noting that the $(MN)^{2/3}$ term is actually attained when taking $c_n=\overline{g(1,n)}$ essentially (see page 7 of \cite{DR}).


\begin{proof}[Proof of Theorem \ref{Bound on 2nd moment}]

 Write $m=m_1m_2^2m_3^3$ where $m_i\equiv 1\pmod 3$, $m_1,m_2$ are squarefree and $(m_1,m_2)=1$. The character $\chi_m$ writes as $\chi_{m_1}\overline{\chi_{m_2}}\chi_{m_3}^3$, where $\chi_{m_3}^3$ is the principal character mod $m_3$. The sum then becomes
\begin{align}\label{L}
&\sum_{N(m)\leq X}\lvert L(1/2+it,\chi_m)\rvert^{2}\nonumber\\
&\hspace{2cm}=\sum_{N(m_3)\leq \sqrt[3]{X}}\,\psum_{N(m_2)\leq \sqrt{\frac{X}{N(m_3)^3}}}\,\psum_{\substack{N(m_1)\leq \frac{X}{N(m_2)^2N(m_3)^3}\\ (m_1,m_2)=1}}\lvert L(1/2+it,\chi_{m_1}\overline{\chi_{m_2}}\chi_{m_3}^3)\rvert^2.
\end{align}
Note that we can write
\begin{align*}
L(1/2+it,\chi_{m_1}\overline{\chi_{m_2}}\chi_{m_3}^3)&=\prod_{\pi\nmid m_3}\left(1-\frac{\chi_{m_1}(\pi)\overline{\chi_{m_2}}(\pi)}{N(\pi)^{1/2+it}}\right)^{-1}\\
&=L(1/2+it,\chi_{m_1}\overline{\chi_{m_2}})\prod_{\pi\mid m_3}\left(1-\frac{\chi_{m_1}(\pi)\overline{\chi_{m_2}}(\pi)}{N(\pi)^{1/2+it}}\right),
\end{align*}
 which reduces the problem to bounding
\[\psum_{N(m_1)\leq Y}\lvert L(1/2+it,\chi_{m_1}\overline{\chi_{m_2})}\rvert^2,\]
for $(m_1,m_2)=1$ and $Y=\frac{X}{N(m_2)^2N(m_3)^3}$. We use the approximate functional equation of Proposition \ref{AFE}, in its trivial balanced case, to bound this second moment by
\begin{align*}
&\leq \psum_{N(m_1)\leq Y}\Bigg\lvert\sum_{n} \frac{\chi_{m_1}(n)\overline{\chi_{m_2}}(n)}{N(n)^{1/2+it}}V_{1/2+it}\left(\frac{N(n)}{\sqrt{N(m_1m_2)}}\right)\\
&+\epsilon(\chi_{m_1}(n)\overline{\chi_{m_2}}(n),1/2+it)\sum_n\frac{\overline{\chi_{m_1}}(n)\chi_{m_2}(n)}{N(n)^{1/2-it}}V_{1/2-it}\left(\frac{N(n)}{\sqrt{N(m_1m_2)}}\right)\Bigg\rvert^2\\
&\leq 3\psum_{N(m_1)\leq Y}\left\lvert\sum_{n} \frac{\chi_{m_1}(n)\overline{\chi_{m_2}(n)}}{N(n)^{1/2+it}}V_{1/2+it}\left(\frac{N(n)}{\sqrt{N(m_1m_2)}}\right)\right\rvert^2,
\end{align*}
since the second sum is just the conjugate of the first. Note that since the sum over $n$ is weighted by $V$, we can truncate it to $N(n)\leq M:=\left(N(m_2)(1+|t|)Y\right)^{1/2+\epsilon}$. Now, in order to remove the $V$ function, we perform partial summation as follows:
\begin{align*}
&\sum_{N(n)\leq M} \frac{\chi_{m_1}(n)\overline{\chi_{m_2}(n)}}{N(n)^{1/2+it}}V_{1/2+it}\left(\frac{N(n)}{\sqrt{N(m_1m_2)}}\right)\\
&=V_{1/2+it}\left(\frac{M}{\sqrt{N(m_1m_2)}}\right)\sum_{N(n)\leq M} \frac{\chi_{m_1}(n)\overline{\chi_{m_2}(n)}}{N(n)^{1/2+it}}-V_{1/2+it}\left(\frac{1}{\sqrt{N(m_1m_2)}}\right)\\
&-\int_{1}^M \frac{1}{\sqrt{N(m_1m_2)}}V'_{1/2+it}\left(\frac{u}{\sqrt{N(m_1m_2)}}\right)\sum_{N(n)\leq u} \frac{\chi_{m_1}(n)\overline{\chi_{m_2}(n)}}{N(n)^{1/2+it}}\, du.
\end{align*}
This done, we can now use the bounds on $V_{1/2+it}$ and its derivatives in Lemma \ref{Bounds on V_s} and conclude that
\[\psum_{N(m_1)\leq Y}\lvert L(1/2+it,\chi_{m_1}\overline{\chi_{m_2})}\rvert^2\ll_{\epsilon} \psum_{N(m_1)\leq Y}\left\lvert \sum_{N(n)\leq M} \frac{\chi_{m_1}(n)\overline{\chi_{m_2}(n)}}{N(n)^{1/2+it}}\right\rvert^2, \]
for any $\epsilon>0$. Writing $n=d^2r$ for $r$ SF, and then using Cauchy-Schwarz and applying Lemma \ref{cubic large sieve} (for the $m_1$ and $r$ sums),  we obtain
\begin{align*}
&\psum_{N(m_1)\leq Y} \left\lvert\sum_{N(d)\leq \sqrt{M}}\frac{\chi_{m_1}\overline{\chi_{m_2}}(d^2)}{N(d)^{1+2it}}\psum_{N(r)\leq \frac{M}{N(d^2)}}\frac{\chi_{m_1}\overline{\chi_{m_2}}(r)}{N(r)^{1/2+it}}\right\rvert^2 \\
&\leq \left( \sum_{ N(d)\leq \sqrt{M}}  \frac{1}{N(d)} \right)\left( \sum_{ N(d)\leq \sqrt{M}}  \frac{1}{N(d)}\psum_{N(m_1)\leq Y} \left\lvert \psum_{ N(r)\leq \frac{M}{N(d^2)}}\frac{\chi_{m_1}\overline{\chi_{m_2}}(r)}{N(r)^{1/2+it}}\right\rvert^2 \right)\\
&\ll \log M \sum_{ N(d)\leq \sqrt{M}}  \frac{1}{N(d)}\left(Y+\frac{M}{N(d)^2}+\left(\frac{YM}{N(d)^2}\right)^{2/3}\right)\left(\frac{YM}{N(d)^2}\right)^{\mathcal{\epsilon}}\psum_{N(r)\leq \frac{M}{N(d)^2}} \frac{1}{N(r)}\\
&\ll Y^{1+\epsilon}N(m_2)^{\epsilon}(1+|t|)^{\epsilon}+Y^{1/2+\epsilon}N(m_2)^{1/2+\epsilon}(1+|t|)^{1/2+\epsilon}+Y^{1+\epsilon}N(m_2)^{1/3+\epsilon}(1+|t|)^{1/3+\epsilon},
\end{align*}
where in the last inequality we used $M=\left(N(m_2)(1+|t|)Y\right)^{1/2+\epsilon}$. Now we replace this bound in \eqref{L} with $Y=\frac{X}{N(m_2)^2N(m_3)^3}$, and sum trivially over $m_2$ and $m_3$, to get
\begin{align*}
\sum_{N(m)\leq X}\lvert L(1/2+it,\chi_m)\rvert^{2}&\ll X^{1+\epsilon}(1+|t|)^{\epsilon}+X^{3/4+\epsilon}(1+|t|)^{1/2+\epsilon}+X^{1+\epsilon}(1+|t|)^{1/3+\epsilon}\\
 &\ll X^{1+\epsilon}(1+|t|)^{1/2+\epsilon},
\end{align*}
as desired.
\end{proof}

 We can now bound $\mathcal{M}_0'(z)$ in \eqref{Error Term}. Using the bound 
 \[\widetilde{f}_z(1/2+it,m,k)\ll\left(1+\frac{1}{2}\sqrt{1+4t^2}\right)^{-E}(1+3^kN(m)/A)^{-E}\ll \frac{3^{-kE}A^E}{N(m)^E}\left(1+\frac{1}{2}\sqrt{1+4t^2}\right)^{-E}, \] 
 we get that
\begin{align*}
\mathcal{M}'_0(z)&\ll_{E,\epsilon} Q^{1/2+\epsilon}A^{E} \sum_{k\geq 0}3^{-k(E+z)}\sum_{\psi \ \mathrm{mod}\ 9}\\
&\times\int_{-\infty}^{\infty}\sum_{\substack{m\equiv 1(\mathrm{mod}\ 3)\\N(m)\leq 3^{-k}A^{1+\epsilon}}}\frac{1}{N(m)^{z+E}} \lvert L(1/2+it,\psi\chi_m)\rvert\left(1+ \frac{1}{2}\sqrt{1+4t^2}\right)^{-E}\,dt,
\end{align*}
for any $E>0$. Using Cauchy-Schwarz to split the sum over $m$, and then applying the bound on the second moment in Theorem \ref{Bound on 2nd moment}, we finally get that for any choice of $E$,
\[ \mathcal{M'}_0(z)\ll_{\epsilon} Q^{1/2+\epsilon}A^{1-z+\epsilon}.\]




\subsection{Main term from cubes}

We now compute $\mathcal{M}_0(z)$ as defined in \eqref{Main Term}. For $a,b=0$ and $m=$\mancube , the Hecke $L$-function has a simple pole at $s=1$. Note that in this case,

\[L(s,\chi_{m^3})=\frac{2}{3}\zeta_{\mathbb{Q}(\omega)}(s)\prod_{\substack{\pi\equiv 1(\mathrm{mod}\ 3)\\\pi\mid m}}\left(1-N(\pi)^{-s}\right)\] 
and \[\sum_{\substack{\ell\equiv 1(\mathrm{mod}\ 3)\\ (\ell,m)=1}}\frac{\mu(\ell)}{N(\ell^2)}=\frac{9}{8\zeta_{\mathbb{Q}(\omega)}(2)}\prod_{\substack{\pi\equiv 1(\mathrm{mod}\ 3)\\ \pi\mid m}}\left(1-N(\pi)^{-2} \right)^{-1}. \] 
Also, from the class number formula, the residue of $\zeta_{\mathbb{Q}(\omega)}(s)$ at $s=1$ is $c_{\omega}=\frac{\pi}{3\sqrt{3}}$. Using these in equation \eqref{Main Term}, we obtain
\begin{align*}
\mathcal{M}_0(z)=\frac{c_{\omega}Q}{2^23\zeta_{\mathbb{Q}(\omega)}(2)}\sum_{k\geq 0}3^{-kz}\sum_{m\equiv 1(\mathrm{mod}\ 3)}\frac{\widetilde{f}_z(1,m^3,k)}{N(m)^{3z}}\prod_{\substack{\pi\equiv 1(\mathrm{mod}\ 3)\\ \pi\mid m}}\left(1+N(\pi)^{-1}\right)^{-1}.
\end{align*}
Now writing $V_{z}$ in its integral form as in Proposition \ref{AFE} and then using the Mellin convolution formula for $\omega(x)$, we get
\begin{align*}
\widetilde{f}(1,m^3,k)&=\int_{0}^{\infty}V_{z}\left(\frac{3^{k-\frac{1}{2}}N(m^3)}{Ax}\right)\omega(x)\,dx\\
&=\frac{1}{2\pi i}\int_{(1)}\left(\frac{A}{3^{k-\frac{1}{2}} N(m^3)}\right)^s\widetilde{\omega}(1+s)\frac{G(s)}{s}g_{z}(s)\,ds.
\end{align*}
Setting \[Z(s):=\sum_{m\equiv 1(\mathrm{mod}\ 3)}N(m)^{-s}\prod_{\substack{\pi\equiv 1(\mathrm{mod}\ 3)\\ \pi\mid m}}\left(1+N(\pi)^{-1}\right)^{-1}, \]
gives
\begin{equation}\label{Z(s)}
\mathcal{M}_0(z)=\frac{c_{\omega}Q}{2^23\zeta_{\mathbb{Q}(\omega)}(2)}\frac{1}{2\pi i}\int_{(1)}\sum_{k\geq 0}3^{-k(s+z)+s/2}A^sZ(3s+3z)\widetilde{\omega}(1+s)\frac{G(s)}{s}g_{z}(s)\,ds.
\end{equation}
Note that $Z(s)$ is holomorphic and bounded for $\Re(s)\geq 1+\delta>1$. Moreover, we can write
\begin{align*}
Z(s)&=\sum_{m\equiv 1(\mathrm{mod}\ 3)}\frac{1}{N(m)^s}\prod_{\substack{\pi\equiv 1(\mathrm{mod}\ 3)\\ \pi\mid m}}\left(1-\frac{1}{N(\pi)+1} \right)=\sum_{m\equiv 1(\mathrm{mod}\ 3)}\frac{\sum_{d\mid m}\frac{\mu(d)}{\sigma(d)}}{N(m)^s}\\
&=\left(1-\frac{1}{3^s}\right)\zeta_{\mathbb{Q}(\omega)}(s)\prod_{\pi\equiv 1(\mathrm{mod}\ 3)}\left(1-\frac{1}{N(\pi)^s(N(\pi)+1)}\right).
\end{align*}

where the Euler product converges absolutely for any $\Re(s)>0$. To find all the poles of $Z(s)$, we further compute

\begin{align*}
Z(s)&=\left(1-\frac{1}{3^s}\right)\left(1-\frac{1}{3^{s+1}}\right)^{-1}\frac{\zeta_{\mathbb{Q}(\omega)}(s)}{\zeta_{\mathbb{Q}(\omega)}(s+1)}\prod_{\pi\equiv 1(\mathrm{mod}\ 3)}\left(1-\frac{1}{N(\pi)^s(N(\pi)+1)}\right)\\
&\hspace{5cm}\times \prod_{\pi\equiv 1(\mathrm{mod}\ 3)}\left(1+\frac{1}{N(\pi)^{s+1}}+O\left(N(\pi)^{-2s-2}\right)\right)\\
&=\left(1-\frac{1}{3^s}\right)\left(1-\frac{1}{3^{s+1}}\right)^{-1}\frac{\zeta_{\mathbb{Q}(\omega)}(s)}{\zeta_{\mathbb{Q}(\omega)}(s+1)}\\
&\hspace{2cm}\times\prod_{\pi\equiv 1(\mathrm{mod}\ 3)}\left(1+\frac{1}{N(\pi)^{s+1}}-\frac{1}{N(\pi)^{s+1}}\frac{1}{(1+N(\pi)^{-1})}+O\left(N(\pi)^{-2s-2}\right)\right)\\
&=\left(1-\frac{1}{3^s}\right)\left(1-\frac{1}{3^{s+1}}\right)^{-1}\frac{\zeta_{\mathbb{Q}(\omega)}(s)}{\zeta_{\mathbb{Q}(\omega)}(s+1)}\prod_{\pi\equiv 1(\mathrm{mod}\ 3)}\left(1+O(N(\pi)^{-s-2})\right).
\end{align*}
The Euler product above is analytic for $\Re(s)>-1$, and the quotient of zeta functions has poles at $s=1$ and in the region where $\Re(s)<0$. Replacing in (\ref{Z(s)}), we can  move the contour of integration depending on the range of $z$.

For $1/3<z\leq 1$, we move the line of integration in \eqref{Z(s)} to $s=-z$ crossing two simple poles at $s=0$ and $s=1/3-z<0$. The new contour contributes $O(QA^{-z})$, while the two poles give $C_z\widetilde{\omega}(1)Q$ at $s=0$, which is the main term, and $D_z\widetilde{\omega}(4/3-z)QA^{1/3-z}$ at $s=1/3-z$, which is the secondary term. The two constants $C_z$ and $D_z$ are given by
\begin{align}\label{Cubic C_z}
C_z&=\left(1+\frac{1}{3^z}+\frac{1}{3^{2z}}\right)\frac{\pi\zeta_{\mathbb{Q}(\omega)}(3z)}{2^23^{7/2}\zeta_{\mathbb{Q}(\omega)}(2)}\prod_{\pi\equiv 1(\mathrm{mod}\ 3)}\left(1-\frac{1}{N(\pi)^{3z}(N(\pi)+1)}\right),
\end{align}
and
\begin{align}\label{Cubic D_z}
D_z
&=\left(\frac{1}{1-\frac{1}{\sqrt[3]{3}}}\right)\frac{2^{z-\frac{4}{3}}\pi^{z+\frac{5}{3}}\Gamma(1/3)}{3^{\frac{35}{6}+\frac{z}{2}}\zeta_{\mathbb{Q}(\omega)}(2)\Gamma(z)}\frac{G(1/3-z)}{1/3-z}\prod_{\pi\equiv 1(\mathrm{mod}\ 3)}\left(1-\frac{1}{N(\pi)(N(\pi)+1)}\right).
\end{align}

For $0<z<1/3$, the same analysis follows but now the pole at $s=1/3-z$ is greater than the pole at $s=0$. Hence, we get the same terms with the same constants in the asymptotic but with the main and secondary terms switched. 

At $z=0$, we capture the pole at $s=1/3$ and then only move the line of integration to $s=\epsilon$ since there might be poles for $\Re(s)<0$. This gives an error term of $O(QA^{\epsilon})$ and a main term of $D_0\widetilde{\omega}(4/3)QA^{1/3}$, where $D_0$ is given in \eqref{Cubic D_z} for $z=0$. Moreover if we assume GRH, we can now move the line further to $s=-1/6+\epsilon$ and capture the pole at $s=0$ which gives the additional secondary term $C_0\widetilde{\omega}(1)Q$, where $C_0$ is given in \eqref{Cubic C_z}, and a better error term of $O(QA^{-1/6+\epsilon})$.

However, at $z=1/3$, when moving the contour of integration in (\ref{Z(s)}) to $s=-1/3$, we now cross a double pole at $s=0$. Again, the new contour contributes $O(QA^{-1/3})$, while the double pole at $s=0$ now gives
\[C_{1/3}\widetilde{\omega}(1)Q\log A+C_2Q,\]
where 
\begin{align}\label{Cubic C_1/3}
C_{1/3}&=\left(\frac{1}{1-\frac{1}{\sqrt[3]{3}}}\right)\frac{\pi^2}{23^6\zeta_{\mathbb{Q}(\omega)}(2)}\prod_{\pi\equiv 1\mathrm(3)}\left(1-\frac{1}{N(\pi)(N(\pi)+1)}\right)
\end{align}
and 
\begin{align}\label{Cubic C_2}
C_2&=\left(\frac{1}{1-\frac{1}{\sqrt[3]{3}}}\right)\prod_{\pi\equiv 1\mathrm(3)}\left(1-\frac{1}{N(\pi)(N(\pi)+1)}\right)\Bigg[\frac{\pi^2\widetilde{\omega}'(1)}{2^23^5\zeta_{\mathbb{Q}(\omega)}(2)}+\frac{\pi^2\widetilde{\omega}(1)}{2^23^5\zeta_{\mathbb{Q}(\omega)}(2)}G'(0)\nonumber\\
&+\frac{\pi^2\widetilde{\omega}(1)}{2^23^5\zeta_{\mathbb{Q}(\omega)}(2)}\left(\frac{\Gamma'(1/3)}{\Gamma(1/3)-\log(2\pi)}\right)+\frac{\pi^2\log(3)\widetilde{\omega}(1)}{2^23^5\zeta_{\mathbb{Q}(\omega)}(2)}\left(\frac{1}{2}-\frac{1}{\sqrt[3]{3}-1} \right)\nonumber\\
&+\frac{\pi\widetilde{\omega}(1)}{2^23^{\frac{5}{2}}\zeta_{\mathbb{Q}(\omega)}(2)}\left(\frac{2\gamma_{\mathbb{Q}(\omega)}+\log 3}{3}+\sum_{\pi\equiv 1\mathrm(3)}\frac{\log N(\pi)}{N(\pi)(N(\pi)+1)-1}\right) \Bigg],
\end{align}
with
\[\gamma_{\mathbb{Q}(\omega)}=\lim_{s\to 1}\left((s-1)\zeta_{\mathbb{Q}(\omega)}(s)\right)' \]
being the analogue of the Euler-Mascheroni constant for $\mathbb{Q}(\omega)$.

\section{The Dual Sum} 

Replacing $\epsilon(z,\chi_q)$ by its value in \eqref{root number}, the dual sum $\mathcal{M}_2(z)$ for any $z\in [0,1]$ can be written as

\[\mathcal{M}_2(z)=\left(2\pi\right)^{2z-1}3^{1/2-z}\frac{\Gamma(1-z)}{\Gamma(z)}\sum_{k\geq 0}3^{-k(1-z)}\sum_{\substack{m\in \mathbb{Z}[\omega]\\  m\equiv 1(\mathrm{mod}\ 3)}}\frac{V_{1-z}\left(3^{k-\frac{1}{2}}\frac{N(m)}{B}\right)}{N(m)^{1-z}}\mathcal{M}_2(m,z) \]
where 
\[\mathcal{M}_2(m,z):=\sum_{\substack{q\in \mathbb{Z}[\omega]\\  q\equiv 1(\mathrm{mod}\ 9)\\ q\, \text{SF}}}\frac{\overline{\left(\frac{m}{q}\right)_3}}{N(q)^{z}}g(1,q)w\left(\frac{N(q)}{Q}\right). \]
Introducing the Mellin transform of $w$, we get 

\begin{equation}\label{M_2 tilde}
\mathcal{M}_2(m,z)=\frac{1}{2\pi i}\int_{(2)}\sum_{\substack{q\in \mathbb{Z}[\omega]\\  q\equiv 1(\mathrm{mod}\ 9)\\ q\, \text{SF}}}\frac{\overline{\left(\frac{m}{q}\right)_3}}{N(q)^{s+z}}g(1,q)\widetilde{w}(s)Q^{s}\, ds.
\end{equation}
Set
\[ \mathcal{G}_2(m,s):=\sum_{\substack{q\in \mathbb{Z}[\omega]\\  q\equiv 1(\mathrm{mod}\ 9)\\ q\, \text{SF}}}\frac{\overline{\left(\frac{m}{q}\right)_3}}{N(q)^{s}}g(1,q).\]
We need to understand the analytic behavior of $\mathcal{G}_2(m,s)$. A similar Dirichlet series appears in the theory of automorphic forms on the metaplectic group first introduced by Kubota in his seminal paper \cite{Kubota1} and later developed by Patterson \cite{Pat1, Pat2}. They obtain the analytic data of the Dirichlet series for the shifted Gauss sums, which is often called the Kubota Dirichlet series given by 
\[\psi(m,s)=\sum_{\substack{n\in \mathbb{Z}[\omega]\\  n\equiv 1(\mathrm{mod}\ 3)}}\frac{g(m,n)}{N(n)^s}, \ \ \ \Re(s)>\frac{3}{2} \]
Therefore, what remains is to relate $\mathcal{G}_2(m,s)$ to $\psi(m,s)$. This will be the goal of the next section and is achieved in Theorem \ref{Removing coprimality}.

\subsection{The Kubota Dirichlet Series}

In the early 1970's, Kubota in a series of papers \cite{Kubota1, Kubota2} introduced the theory of metaplectic forms, which are functions automorphic under certain metaplectic groups associated to a number field that contains the $n^{th}$ roots of unity. For the case $n=3$, which is what concerns us here, he introduced the cubic metaplectic theta function $\theta(w)$ on the hyperbolic 3-space $\mathbb{H}^3=\{w=(z,v)\in \mathbb{C}\times \mathbb{R}^+\}$, given by
\[\theta(w)=\frac{3^{5/2}}{2}v^{2/3}+\sum_{\mu\in \lambda^{-3}\mathbb{Z}[\omega]}\tau(\mu)vK_{\frac{1}{3}}(4\pi|\mu|v)e(\mu z) \]
where $K_{\frac{1}{3}}$ is the modified Bessel function of the second kind. The coefficients $\tau(\mu)$ were later explicitly computed by Patterson \cite[Thm 8.1]{Pat1}, and are given in terms of cubic Gauss sums.  For any $\mu\in \lambda^{-3}\mathbb{Z}[\omega]$, the coefficients are
\begin{equation}\label{tau(r)}
\tau(\mu)=
\begin{cases}
\overline{g(\lambda^2,c)}\lvert\frac{d}{c} \rvert 3^{\frac{n}{2}+2} & \text{if} \ \mu=\pm\lambda^{3n-4}cd^3 \text{for} \ n\geq 2\\
e^{\frac{-2\pi i}{9}}\overline{g(\omega\lambda^2,c)}\lvert\frac{d}{c} \rvert 3^{\frac{n}{2}+2} & \text{if} \ \mu=\pm\omega\lambda^{3n-4}cd^3 \text{for} \ n\geq 2\\
e^{\frac{2\pi i}{9}}\overline{g(\omega^2\lambda^2,c)}\lvert\frac{d}{c} \rvert 3^{\frac{n}{2}+2} & \text{if} \ \mu=\pm\omega^{2}\lambda^{3n-4}cd^3 \text{for} \ n\geq 2\\
\overline{g(1,c)}\lvert\frac{d}{c} \rvert 3^{\frac{n+5}{2}} & \text{if} \ \mu=\pm\lambda^{3n-3}cd^3 \text{for} \ n\geq 1\\
0 & \text{otherwise},
\end{cases}
 \end{equation}
 where $c,d\in \mathbb{Z}[\omega]$ satisfy $c,d\equiv 1\pmod{3}$ and $\mu^2(c)=1$.
 
 Patterson also studied the analytic behavior of the Dirichlet series $\psi(r,s)$ and showed how its residues are related to the coefficients $\tau(\mu)$. We summarize these in the next lemma, where it is a consequence of Theorems 6.1 and 9.1 of \cite{Pat1}, together with \cite{Pat2}.

\begin{lemma}(Poles and residues)\label{Patterson}
For any $0\neq r\in \mathbb{Z}[\omega]$, the function $\psi(r,s)$ has meromorphic continuation to the complex plane. It is holomorphic in the region $\Re(s)>1$ except possibly for a simple pole at $s=4/3$. The residue at this pole equals
\[\underset{s=\frac{4}{3}}{\text{Res}}\, \psi(r,s)=\frac{c_0\tau(r)}{N(r)^{1/6}} \ \ \text{with} \ \ c_0=\frac{(2\pi)^{5/3}}{3^{9/2}8\Gamma(2/3)\zeta_{\mathbb{Q}(\omega)}(2)},\]
where $\tau(r)$ is given in \eqref{tau(r)}.
\end{lemma}

We will also need the following bound on the $\psi$ function, which is the convexity bound, proved for $r\equiv 1\pmod 3$ by Heath-Brown and Patterson \cite[Lemma 4]{HBP79}, and extended to any $r\in \mathbb{Z}[\omega]$ by David, De Faveri, Dunn, and Stucky \cite[Lemma 6.3]{chantal2024}.

\begin{lemma}\label{Patterson2}(Convexity bound)
Let $0\neq r\in \mathbb{Z}[\omega]$ and $\epsilon>0$. For $s=\sigma+it$ with $1+\epsilon\leq \sigma\leq \frac{3}{2}+\epsilon$ and $|s-\frac{4}{3}|>\epsilon$, we have the bound
\[\psi(r,s)\ll_{\epsilon} N(r)^{\frac{3}{4}-\frac{\sigma}{2}+\epsilon}(1+t^2)^{\frac{3}{2}-\sigma+\epsilon}. \]
\end{lemma}

 \subsubsection{Relating $\mathcal{G}_2(r,s)$ to $\psi(r,s)$.}
 
The work in this section is similar to that of Heath-Brown and Patterson \cite[Lemma 3]{HBP79}, but here we work with any Eisenstein integer and not just those $\equiv 1\pmod 3$. We need to keep track of the factors $\omega^a\lambda^b$ to obtain the contribution of all the different residues in order to observe the cancellation in the main terms later.

 Using property (\ref{1}) of Lemma \ref{properties of g} for the Gauss sums $g(r,n)$, we have that
\[
\overline{\left(\frac{r}{n}\right)_3}g(1,n)=\begin{cases}
			g(r,n) & \text{if $(r,n)=1$}\\
            0 & \text{otherwise}.
		 \end{cases}
 \]
 But for $(n,r)=1$, we saw that $g(r,n)=0$ if $n$ is not SF, which means that we can now remove the SF condition in our sum to get
\[\mathcal{G}_2(r,s)=\sum_{\substack{n\in \mathbb{Z}[\omega]\\  n\equiv 1(\mathrm{mod}\ 9)\\ (n,r)=1}}\frac{g(r,n)}{N(n)^{s}}. \]
We want the sum to run over all $n\in \mathbb{Z}[\omega]$ such that $n\equiv 1 \pmod {3}$. We make use of Lemma \ref{detect mod 9} again, to write
 
 \begin{equation}\label{introducing a,b}
 \mathcal{G}_2(r,s)=\frac{1}{9}\sum_{a,b=0}^2\sum_{\substack{n\in \mathbb{Z}[\omega]\\  n\equiv 1(\mathrm{mod}\ 3)\\ (n,r)=1}}\frac{g(r,n)\overline{\chi_n(\omega^a\lambda^b)}}{N(n)^{s}}=\frac{1}{9}\sum_{a,b=0}^2\sum_{\substack{n\in \mathbb{Z}[\omega]\\  n\equiv 1(\mathrm{mod}\ 3)\\ (n,r)=1}}\frac{g(\omega^a\lambda^br,n)}{N(n)^{s}}.  
 \end{equation}
 Now set
\[\mathcal{G}_2(r,s,a,b):=\sum_{\substack{n\in \mathbb{Z}[\omega]\\  n\equiv 1(\mathrm{mod}\ 3)\\ (n,r)=1}}\frac{g(\omega^a\lambda^br,n)}{N(n)^{s}}, \]
 and proceed by writing $r=r_1r_2^2r_3^3$ where $r_i\equiv 1 \pmod 3$, $r_1,r_2$ are SF and $(r_1,r_2)=1$. Denote also by
\[ r_3^*=\prod_{\substack{\pi\mid r_3\\ \pi\nmid r_1r_2}}\pi,\] 
where the product is over primes $\pi\in \mathbb{Z}(\omega)$ that are $\equiv 1\pmod 3$. Using property (\ref{1}) again, we get

\begin{align*}
\mathcal{G}_2(r,s,a,b)&=\sum_{\substack{n\equiv 1(\mathrm{mod}\ 3)\\ (n,r_1r_2^2)=1\\ (n,r_3)=1}}\overline{\left(\frac{r_3^3}{n}\right)_3}\frac{g(\omega^a\lambda^br_1r_2^2,n)}{N(n)^{s}}\\
&=\sum_{\substack{n\equiv 1(\mathrm{mod}\ 3)\\ (n,r_1r_2^2)=1}}\frac{g(\omega^a\lambda^br_1r_2^2,n)}{N(n)^{s}}\sum_{\substack{d\mid n\\ d\mid r_3^*}}\mu(d)\\
&=\sum_{ d\mid r_3^*}\frac{\mu(d)}{N(d)^{s}}\sum_{\substack{n\equiv 1(\mathrm{mod}\ 3)\\ (dn,r_1r_2^2)=1}}\frac{g(\omega^a\lambda^br_1r_2^2,dn)}{N(n)^{s}}
\end{align*}
Since $(dn,r_1r_2^2)=1$, then $g(\omega^a\lambda^br_1r_2^2,dn)=0$ unless $dn$ is SF. So we can restrict to $dn$ SF, which implies $(d,n)=1$. Therefore, the condition $(dn,r_1r_2^2)=1$ becomes $(n,dr_1r_2^2)=1$ and we can then use property (\ref{2}) to get 
\[\mathcal{G}_2(r,s,a,b)=\sum_{ d\mid r_3^*}\frac{\mu(d)}{N(d)^{s}}g(\omega^a\lambda^br_1r_2^2,d)\mathcal{G}_2(r,s,a,b,d) \]
where
\begin{equation}\label{(m,s,a,b,d)}
\mathcal{G}_2(r,s,a,b,d):=\sum_{\substack{n\equiv 1(\mathrm{mod}\ 3)\\ (n,dr_1r_2^2)=1}}\frac{g(\omega^a\lambda^bdr_1r_2^2,n)}{N(n)^{s}}.
\end{equation}

Now, we make use of the local properties of the Gauss sums in Lemma \ref{local properties of g}. The first step concerns the primes dividing $r_2$. This results in  the following lemma. 

\begin{lemma}\label{dr_1} If $a,b\in \{0,1,2\},$ $s\in \mathbb{C}$, and $r,d\in \mathbb{Z}[\omega]$ with $r,d\equiv 1\pmod{3}$, then
\[ \mathcal{G}_2(r,s,a,b,d)=\prod_{\pi\mid r_2}(1-N(\pi)^{2-3s})^{-1}\sum_{(n,dr_1)=1}\frac{g(\omega^a\lambda^bdr_1r_2^2,n)}{N(n)^s}.\]
\end{lemma}

\begin{proof} We want to induct on the number of primes that divide $r_2$. For that purpose, write $r_2=\beta\pi$ where $\pi$ is a prime and $\beta\in \mathbb{Z}[\omega]$, and consider 
\begin{align}\label{S_1}
\mathcal{S}_1&:=\sum_{(n,dr_1\pi)=1}\frac{g(\omega^a\lambda^bdr_1\beta^2\pi^2,n)}{N(n)^s}\nonumber\\
&=\sum_{(n,dr_1)=1}\frac{g(\omega^a\lambda^bdr_1\beta^2\pi^2,n)}{N(n)^s}-\sum_{\substack{(n,dr_1)=1\\ \pi\mid n}}\frac{g(\omega^a\lambda^bdr_1\beta^2\pi^2,n)}{N(n)^s}.
\end{align}
Let $\mathcal{S}_{1,1}$ be the second sum on the right hand side of \eqref{S_1}. Writing $n=\pi^kn'$,  we obtain the term $g(\omega^a\lambda^b dr_1\beta^2\pi^2,\pi^kn')$ which splits as
\[g(\omega^a\lambda^b dr_1\beta^2\pi^2,\pi^kn')=g(\omega^a\lambda^bdr_1\beta^2\pi^{2+k},n')g(\omega^a\lambda^bdr_1\beta^2\pi^2,\pi^k), \]
by property \eqref{1}. Using Lemma \ref{local properties of g}, we deduce that $g(\omega^a\lambda^bdr_1\beta^2\pi^2,\pi^k)$ vanishes when $k\neq 3$ and equals $-N(\pi^2)$ when $k=3$, for any $a,b\in\{0,1,2\}$. Thus $\mathcal{S}_{1,1}$ is restricted to the terms $n=\pi^3n'$, and thus
\begin{align*}
\mathcal{S}_{1,1}&=\sum_{\substack{(\pi^3n,dr_1)=1\\ \pi\nmid n}}N(\pi^2)\frac{g(\omega^a\lambda^bdr_1\beta^2\pi^{2+3},n)}{N(n)^sN(\pi)^{3s}}\\
&=\sum_{(n,\pi dr_1)=1}N(\pi)^{-3s+2}\frac{g(\omega^a\lambda^bdr_1\beta^2\pi^{2+3},n)}{N(n)^s}=N(\pi)^{-3s+2}\mathcal{S}_1.
\end{align*}
Therefore, \[\mathcal{S}_1=(1-N(\pi)^{2-3s})^{-1}\sum_{(n,dr_1)=1}\frac{g(dr_1\beta^2\pi^2,n)}{N(n)^s}. \]
Now write $r_2=\beta\pi_1\pi_2$ and consider

\begin{align}\label{S_2}
\mathcal{S}_2&:=\sum_{(n,dr_1\pi_1\pi_2)=1}\frac{g(\omega^a\lambda^bdr_1\beta^2\pi_1^2\pi_2^2,n)}{N(n)^s}\nonumber\\
&=\sum_{(n,dr_1\pi_1)=1}\frac{g(\omega^a\lambda^bdr_1\beta^2\pi_1^2\pi_2^2,n)}{N(n)^s}-\sum_{\substack{(n,dr_1\pi_1)=1\\ \pi_2\mid n}}\frac{g(\omega^a\lambda^bdr_1\beta^2\pi_1^2\pi_2^2,n)}{N(n)^s} 
\end{align}
Again, call the second sum in \eqref{S_2} $\mathcal{S}_{2,1}$ and, as we did above, we similarly get
\begin{equation}\label{S_2,2}
\mathcal{S}_2=(1-N(\pi_2)^{2-3s})^{-1}\sum_{(n,dr_1\pi_1)=1}\frac{g(\omega^a\lambda^bdr_1\beta^2\pi_1^2\pi_2^2,n)}{N(n)^s}. 
\end{equation}
Now, the sum in the right hand side of \eqref{S_2,2} is just a copy of $\mathcal{S}_1$ with $\beta\pi_1$ in place of $\beta$, and so we use that case to get
\[ \mathcal{S}_2=(1-N(\pi_2)^{2-3s})^{-1}(1+N(\pi_1)^{2-3s})^{-1}\sum_{(n,dr_1)=1}\frac{g(\omega^a\lambda^bdr_1\beta^2\pi_1^2\pi_2^2,n)}{N(n)^s}.\]
Exhausting all the primes that divide $r_2$, and using induction on the number of these primes, we finally arrive at

\begin{equation*}
\sum_{(n,dr_1r_2)=1}\frac{g(\omega^a\lambda^bdr_1r_2^2,n)}{N(n)^s}=\prod_{\pi\mid r_2}(1-N(\pi)^{2-3s})^{-1}\sum_{(n,dr_1)=1}\frac{g(\omega^a\lambda^bdr_1r_2^2,n)}{N(n)^s},
\end{equation*}
as desired.
\end{proof}

That done, we now turn to studying $dr_1$ locally at its primes and finally prove the following:

\begin{theorem}\label{Removing coprimality}
  For $a,b\in \{0,1,2\},$ $s\in \mathbb{C}$, and $r,d\in \mathbb{Z}[\omega]$ with $r,d\equiv 1\pmod{3}$, we have that
  \begin{align*}
\mathcal{G}_2(r,s,a,b,d)&=\prod_{\pi\mid r_2}(1-N(\pi)^{2-3s})^{-1}\prod_{\pi\mid dr_1}(1-N(\pi)^{2-3s})^{-1}\\
&\times \sum_{\alpha\mid dr_1}\mu(\alpha)N(\alpha)^{1-2s}\overline{g(\omega^a\lambda^bdr_1r_2^2/\alpha,\alpha)}\psi(\omega^a\lambda^bdr_1r_2^2/\alpha,s).
  \end{align*}

\end{theorem}

\begin{proof}

Write $dr_1=\alpha\pi$ for a prime $\pi\equiv 1\pmod 3$, and consider the sum
\[H_1:=\sum_{(n,\alpha)=1}\frac{g(\omega^a\lambda^br_2^2,n)}{N(n)^s}=\sum_{k=0}^{\infty}\sum_{(n',\pi\alpha)=1}\frac{g(\omega^a\lambda^br_2^2,\pi^kn')}{N(\pi)^sN(n')^s} \]
where we wrote $n=n'\pi^k$ with $(n',\pi)=1$. Since $\pi\nmid r_2$, we use Lemma \ref{local properties of g} to get that the only nonzero terms are $k=0,1$ which contribute
\[ \sum_{(n',\pi\alpha)=1}\frac{g(\omega^a\lambda^br_2^2,n')}{N(n')^s}+\sum_{(n',\pi\alpha)=1}\frac{g(\omega^a\lambda^br_2^2,\pi n')}{N(\pi)^sN(n')^s}\]
As before, in the second sum use $g(\omega^a\lambda^br_2^2,\pi n')=g(\omega^a\lambda^br_2^2\pi,n')g(\omega^a\lambda^br_2^2,\pi)$ to get
\begin{equation}\label{H_1}
H_1=\sum_{(n,\pi\alpha)=1}\frac{g(\omega^a\lambda^br_2^2,n)}{N(n)^s}+N(\pi)^{-s}g(\omega^a\lambda^br_2^2,\pi)\sum_{(n,\pi\alpha)=1}\frac{g(\omega^a\lambda^br_2^2\pi, n)}{N(n)^s}.
\end{equation}
Now consider another sum 
\[H_2:=\sum_{(n,\alpha)=1}\frac{g(\omega^a\lambda^b\pi r_2^2,n)}{N(n)^s}=\sum_{k=0}^{\infty}\sum_{(n',\pi\alpha)=1}\frac{g(\omega^a\lambda^b\pi r_2^2,\pi^kn')}{N(\pi)^sN(n')^s}. \]
Again using Lemma \ref{local properties of g}, we see that the only terms that remain are for $k=0,2$; and we again write  $g(\omega^a\lambda^b\pi r_2^2,\pi^2 n')=g(\omega^a\lambda^br_2^2\pi,\pi^2)g(\omega^a\lambda^b\pi^3 r_2^2,n')$ to get

\begin{equation}\label{H_2}
H_2=\sum_{(n,\pi\alpha)=1}\frac{g(\omega^a\lambda^b\pi r_2^2,n)}{N(n)^s}+N(\pi)^{1-2s}\overline{g(\omega^a\lambda^br_2^2,\pi)}\sum_{(n,\pi\alpha)=1}\frac{g(\omega^a\lambda^br_2^2,n)}{N(n)^s}.
\end{equation}
Now using both (\ref{H_1}) and (\ref{H_2}), and the fact that $\lvert g(\omega^a\lambda^br_2^2,\pi)\rvert^2=N(\pi)$, one obtains

\begin{align}\label{*}
&\sum_{(n,\pi\alpha)=1}\frac{g(\omega^a\lambda^b\pi r_2^2,n)}{N(n)^s}\left(1-N(\pi)^{2-3s}\right)\nonumber\\
&=\sum_{(n,\alpha)=1}\frac{g(\omega^a\lambda^b\pi r_2^2,n)}{N(n)^s}-N(\pi)^{1-2s}\overline{g(\omega^a\lambda^br_2^2,\pi)}\sum_{(n,\alpha)=1}\frac{g(\omega^a\lambda^br_2^2,n)}{N(n)^s}.
\end{align}
We want to continue in this fashion and perform induction on the number of primes dividing $dr_1$. For the induction hypothesis, assume 

\begin{equation}\label{+}
\sum_{(n,\alpha)=1}\frac{g(\omega^a\lambda^b\alpha r_2^2,n)}{N(n)^s}\prod_{p\mid \alpha}\left(1-N(p)^{2-3s}\right)=\sum_{a\mid \alpha}\mu(a)N(a)^{1-2s}\overline{g(\omega^a\lambda^b\alpha r_2^2/a,a)}\psi(\alpha r_2^2/a,s),
\end{equation}
and denote by
\[J_1:=\sum_{(n,\alpha\pi)=1}\frac{g(\omega^a\lambda^b\alpha\pi r_2^2,n)}{N(n)^s}. \]
Use (\ref{*}) with $\alpha r_2^2$ in place of $r_2^2$ to get that
\[J_1=\left(1-N(\pi)^{2-3s}\right)^{-1}\left[\sum_{(n,\alpha)=1}\frac{g(\omega^a\lambda^b\alpha r_2^2\pi,n)}{N(n)^s}-N(\pi)^{1-2s}\overline{g(\omega^a\lambda^b\alpha r_2^2,\pi)}\sum_{(n,\alpha)=1}\frac{g(\omega^a\lambda^b\alpha r_2^2,n)}{N(n)^s}\right] \]
where we can now use the inductive step (\ref{+}) for these two sums, with $r_2^2$ replaced by $\pi r_2^2$ in the first sum, to obtain that $J_1$ equals
\begin{align*}
&\prod_{p\mid \alpha\pi}\left(1-N(p)^{2-3s}\right)^{-1}\Bigg[\sum_{a\mid \alpha}\mu(a)N(a)^{1-2s}\overline{g(\omega^a\lambda^b\alpha\pi r_2^2/a,a)}\psi(\alpha\pi r_2^2/a,s)\\
&-N(\pi)^{-2s+1}\overline{g(\omega^a\lambda^b\alpha r_2^2,\pi)}\sum_{a\mid \alpha}\mu(a)N(a)^{1-2s}\overline{g(\omega^a\lambda^b\alpha r_2^2/a,a)}\psi(\alpha r_2^2/a,s)\Bigg].
\end{align*}
Finally, using $g(\omega^a\lambda^b\alpha r_2^2,\pi)g(\omega^a\lambda^b\alpha r_2^2/a,a)=g(\omega^a\lambda^b\alpha r_2^2/a,a\pi)$ and $\mu(\pi)=-1$, and recalling that $dr_1=\alpha\pi$, we can plug this in Lemma \ref{dr_1} and the result will then follow.
\end{proof}

Recalling that 
\begin{equation}\label{mathcal G}
\mathcal{G}_2(r,s)=\frac{1}{9}\sum_{a,b=0}^2\sum_{d|r_3^*}\frac{\mu(d)}{N(d)^s}g(\omega^a\lambda^br_1r_2^2,d)\mathcal{G}_2(r,s,a,b,d),
\end{equation}
Theorem \ref{Removing coprimality} will thus relate $\mathcal{G}_2(r,s)$ to the Kubota Dirichlet series $\psi(r,s)$ as needed.

\subsection{Asymptotic for the Dual Sum}

Now that we have Theorem \ref{Removing coprimality}, we can use the two Lemmata \ref{Patterson} and \ref{Patterson2} to deduce similar results for $\mathcal{G}_2(r,s)$. Consequently, we move the line of integration in \eqref{M_2 tilde} to the line $s=1-z+\epsilon$, crossing a simple pole at $s=4/3-z$ only, which comes from the $\psi$ function. Denote by $\mathcal{M}_2'(z)$ the contribution of $\mathcal{M}_2(z)$ from this pole, and by $\mathcal{M}_2''(z)$ the contribution of the new contour.

\begin{corollary}\label{M_2''}
For any $z\in [0,1]$ and $\epsilon>0$, we have the bound
\[\mathcal{M}_2''(z)\ll Q^{1-z+\epsilon}B^{\frac{1}{4}+z+\epsilon}. \]
\end{corollary}
\begin{proof}

Recall that $r=r_1r_2^2r_3^3$. Applying Theorem \ref{Removing coprimality} into \eqref{mathcal G} to bound the integral \eqref{M_2 tilde} on the line $s=1-z+\epsilon$, and replacing in $\mathcal{M}_2(z)$ will give
\begin{align*}
\mathcal{M}_2''(z)&\ll_z\sum_{k\geq 0}3^{-k(1-z)}\sum_{N(r)\leq 3^{-k+\frac{1}{2}}B^{1+\epsilon}}\frac{\left\lvert V_{1-z}\left(3^{k-\frac{1}{2}}\frac{N(r)}{B}\right)\right\rvert}{N(r)^{1-z}}\int_{-\infty}^{\infty}\Bigg\lvert \prod_{\pi\mid r_2}(1-N(\pi)^{-1-\epsilon-it})^{-1}\\
&\times\prod_{\pi\mid dr_1}(1-N(\pi)^{-1-\epsilon-it})^{-1}\sum_{ d\mid r_3^*}\frac{g(r_1r_2^2,d)}{N(d)^{1+\epsilon+it}}\sum_{a\mid dr_1}\frac{\overline{g(dr_1r_2^2/a,a)}}{N(a)^{1+\epsilon+2it}}\widetilde{\omega}(1-z+\epsilon+it)\\
&\hspace{6cm}\times\psi\left(dr_1r_2^2/a,1+\epsilon+it\right)Q^{1-z+\epsilon+it}\bigg\rvert\,dt
\end{align*}
We use Lemma \ref{Bounds on V_s} to bound the $V_{1-z}$ function by
\[V_{1-z}\left(3^{k-\frac{1}{2}}\frac{N(r)}{B}\right)\ll_E \left(1+\frac{3^{k-\frac{1}{2}}N(r)}{B\sqrt{(4-z)(5-z)}}\right)^{-E}\ll_{E,z} \left(3^{k-\frac{1}{2}}\frac{N(r)}{B}\right)^{-E},\]
for an $E>0$ that we will choose at the end. Hence
\begin{align*}
\mathcal{M}_2''(z)&\ll_{z,E} Q^{1-z+\epsilon}B^E\sum_{k\geq 0}3^{-k(1-z+E)+\frac{E}{2}}\sum_{N(r)\leq 3^{-k+\frac{1}{2}} B^{1+\epsilon}}\frac{1}{N(r)^{1-z+E}}\int_{-\infty}^{\infty}\Bigg\lvert
\sum_{ d\mid r_3^*}\frac{g(r_1r_2^2,d)}{N(d)^{1+\epsilon+it}}\\
&\times\sum_{a\mid dr_1}\frac{\overline{g(dr_1r_2^2/a,a)}}{N(a)^{1+\epsilon+2it}}\widetilde{\omega}(1-z+\epsilon+it)\,\psi\left(dr_1r_2^2/a,1+\epsilon+it\right)\big\rvert\,dt.
\end{align*}
Bounding the Gauss sums trivially and using Lemma \ref{Patterson2} to bound the $\psi$ function, we get
\begin{align*}
\mathcal{M}_2''(z)&\ll_{z,E,\epsilon} Q^{1-z+\epsilon}B^E\sum_{k\geq 0}3^{-k(1-z+E)+\frac{E}{2}}\sum_{N(r) \leq 3^{-k+\frac{1}{2}}B^{1+\epsilon}}\frac{N(r_1)^{1/4+\epsilon}N(r_2)^{1/2+\epsilon}}{N(r_1r_2^2r_3^3)^{1-z+E}}\\
& \hspace{1cm}\times\sum_{ d\mid r_3^*}\frac{1}{N(d)^{1/4+\epsilon}}\sum_{a\mid dr_1}\frac{1}{N(a)^{3/4+\epsilon}}\int_{-\infty}^{\infty} |\widetilde{\omega}(1-z+\epsilon+it)|(1+t^2)^{1/2} \,dt\\
&\ll_{z,\epsilon} Q^{1-z+\epsilon}B^{\frac{1}{4}+z+\epsilon}
\end{align*}
for any $\epsilon>0$, upon choosing $E=1/4+z+\epsilon$ such that the $r$-sum is bounded by $B^{\epsilon}$.
\end{proof}

For the main term, we use Theorem \ref{Removing coprimality} in \eqref{mathcal G}, and then replace in \eqref{M_2 tilde} and $\mathcal{M}_2(z)$ for any $z\in [0,1]$ to get that the contribution of the pole at $s=4/3-z$ is
\begin{align*}
\mathcal{M}_2'(z)&=\left(\frac{2}{\sqrt{3}}\right)^{2z-1}\frac{\pi^{2z-1}\Gamma(1-z)}{\Gamma(z)}Q^{\frac{4}{3}-z}\widetilde{\omega}\left(\frac{4}{3}-z\right)\sum_{k\geq 0}3^{-k(1-z)}\sum_{r}\frac{V_{1-z}\left(3^{k-\frac{1}{2}}\frac{N(r)}{B}\right)}{N(r)^{1-z}}\\
&\times\frac{1}{9}\sum_{a,b=0}^2\sum_{d\mid r_3^*}\frac{\mu(d)}{N(d)^{4/3}}g(\omega^a\lambda^br_1,d)\prod_{\pi|dr_1}\left(1-N(\pi)^{-2}\right)^{-1}\\
&\times \sum_{c\mid dr_1}\frac{\mu(c)}{N(c)^{5/3}}\overline{g(\omega^a\lambda^bdr_1/c,c)}\,\underset{s=\frac{4}{3}}{\text{Res}}\,\psi(\omega^a\lambda^bdr_1/c,s).
\end{align*}
By Lemma \ref{Patterson},
\begin{align*}
\mathcal{M}_2'(z)&=\frac{\pi^{2z+\frac{2}{3}} \Gamma(1-z)Q^{\frac{4}{3}-z}\widetilde{\omega}\left(\frac{4}{3}-z\right)}{2^{\frac{7}{3}-2z}3^{z+3}\Gamma(2/3)\Gamma(z)\zeta_{\mathbb{Q}(\omega)}(2)}\sum_{k\geq 0}3^{-k(1-z)}\sum_{r}\frac{V_{1-z}\left(3^{k-\frac{1}{2}}\frac{N(r)}{B}\right)}{N(r)^{1-z}}\\
&\times\sum_{a,b=0}^2\delta_{a,b}\sum_{d\mid r_3^*}\frac{\mu(d)}{N(d)^{4/3}}\prod_{\pi|dr_1}\left(1-N(\pi)^{-2}\right)^{-1}\\
&\times\sum_{c\mid dr_1}\frac{\mu(c)}{N(c)^{5/3}}g(\omega^a\lambda^br_1,d)\overline{g(\omega^a\lambda^bdr_1/c,c)}\overline{g(\omega^a\lambda^b,dr_1/c)}N(dr_1/c)^{-2/3},
\end{align*}
 where $$\delta_{a,b}=\begin{cases}
1 & a=b=0\\
 3^{-1/3}& a=0, b=2\\
 3^{-1/3}e^{-2\pi i/9}& a=1,b=2\\
 3^{-1/3}e^{2\pi i/9}& a=b=2\\
 0&\text{otherwise}.
\end{cases}$$
Now from property (\ref{2}), we have the two equations
\[ g(\omega^a\lambda^b,dr_1)=g(\omega^a\lambda^b,dr_1/c)g(\omega^a\lambda^bdr_1/c,c),\]
and 
\[g(\omega^a\lambda^b,dr_1)=g(\omega^a\lambda^br_1,d)g(\omega^a\lambda^b,r_1). \]
Consequently, the three Gauss sums can be nicely combined to obtain
\begin{align*}
\mathcal{M}_2'(z)&=\frac{\pi^{2z+\frac{2}{3}} \Gamma(1-z)Q^{\frac{4}{3}-z}\widetilde{\omega}\left(\frac{4}{3}-z\right)}{2^{\frac{7}{3}-2z}3^{z+3}\Gamma(2/3)\Gamma(z)\zeta_{\mathbb{Q}(\omega)}(2)}\sum_{a,b=0}^2\delta_{a,b}\sum_{k\geq 0}3^{-k(1-z)}\sum_{r}\frac{V_{1-z}\left(3^{k-\frac{1}{2}}\frac{N(r)}{B}\right)}{N(r)^{1-z}}\\
&\times\frac{\overline{g(\omega^a\lambda^b,r_1)}}{N(r_1)^{2/3}}\sum_{d\mid r_3^*}\frac{\mu(d)}{N(d)}\prod_{\pi|dr_1}\left(1-N(\pi)^{-2}\right)^{-1}\sum_{c\mid dr_1}\frac{\mu(c)}{N(c)}.
\end{align*}
Now note that
\[\sum_{c\mid dr_1}\frac{\mu(c)}{N(c)}=\prod_{\pi\mid dr_1}\left(1-\frac{1}{N(\pi)}\right). \]
Therefore,
 \[\prod_{\pi|dr_1}\left(1-N(\pi)^{-2}\right)^{-1}\prod_{\pi\mid dr_1}\left(1-\frac{1}{N(\pi)}\right)=\prod_{\pi|dr_1}\left(1+\frac{1}{N(\pi)}\right)^{-1}, \]
 and since $(d,r_1)=1$, we get that
 \begin{align*}
 \sum_{d\mid r_3^*}\frac{\mu(d)}{N(d)}\prod_{\pi|d}\left(1+\frac{1}{N(\pi)}\right)^{-1}\prod_{\pi|r_1}\left(1+\frac{1}{N(\pi)}\right)^{-1}&=\prod_{\pi\mid r_3^*}\left(1-\frac{1}{N(\pi)+1}\right)\prod_{\pi|r_1}\left(1+\frac{1}{N(\pi)}\right)^{-1}\\
 &=\prod_{\pi\mid r_1r_3}\left(1+\frac{1}{N(\pi)} \right)^{-1}.
 \end{align*}
Replacing in $\mathcal{M}_2'(z)$, and then writing $V_{1-z}$ in its integral form as in Proposition \ref{AFE}, and recalling that $r=r_1r_3^3$, we get
\begin{align}\label{int of  MT(z)}
\mathcal{M}_2'(z)&=\frac{\pi^{2z+\frac{2}{3}} \Gamma(1-z)Q^{\frac{4}{3}-z}\widetilde{w}\left(\frac{4}{3}-z\right)}{2^{\frac{7}{3}-2z}3^{z+3}\Gamma(2/3)\Gamma(z)\zeta_{\mathbb{Q}(\omega)}(2)}\sum_{a,b=0}^2\frac{\delta_{a,b}}{2\pi i}\sum_{k\geq 0}\int_{(3)}3^{k(-s-1+z)+s/2}B^s\frac{G(s)}{s}g_{1-z}(s)\nonumber\\
&\times \sum_{r_3}\frac{1}{N(r_3)^{3s-3z+3}}\sum_{r_1 SF}\frac{\overline{g(\omega^a\lambda^b,r_1)}}{N(r_1)^{s-z+5/3}}\prod_{\pi\mid r_1r_3}\left(1+\frac{1}{N(\pi)} \right)^{-1} \,ds.
\end{align}
We first write
\begin{align*}
\prod_{\pi\mid r_1r_3}\left(1+\frac{1}{N(\pi)} \right)^{-1}&=\prod_{\pi\mid r_3}\left(1+\frac{1}{N(\pi)} \right)^{-1}\prod_{\substack{\pi\mid r_1\\\pi\nmid r_3}}\left(1-\frac{1}{N(\pi)+1} \right)\\
&=\prod_{\pi\mid r_3}\left(1+\frac{1}{N(\pi)} \right)^{-1}\sum_{\substack{d\mid r_1\\(d,r_3)=1}}\frac{\mu(d)}{\sigma(d)}.
\end{align*}
Next, we exchange the $r_1$ and $d$ sums, and then use $g(\omega^a\lambda^b,dr_1)=g(\omega^a\lambda^b,d)g(\omega^a\lambda^bd,r_1)$ as in property (\ref{2}), to deduce that
\begin{align}\label{exchanging r_1,d}
&\sum_{r_1 SF}\frac{\overline{g(\omega^a\lambda^b,r_1)}}{N(r_1)^{s-z+5/3}}\prod_{\pi\mid r_1r_3}\left(1+\frac{1}{N(\pi)} \right)^{-1}\nonumber\\
&=\prod_{\pi\mid r_3}\left(1+\frac{1}{N(\pi)} \right)^{-1}\sum_{\substack{(d,r_3)=1}}\frac{\mu(d)\overline{g(\omega^a\lambda^b,d)}}{N(d)^{s-z+\frac{5}{3}}\sigma(d)}\sum_{\substack{r_1 SF\\(r_1,d)=1}}\frac{\overline{g(\omega^a\lambda^bd,r_1)}}{N(r_1)^{s-z+5/3}}.
\end{align}
We are now able to employ Theorem \ref{Removing coprimality} yet again (but now with $r_2=1$), to get that
\begin{align}\label{2nd application of Thm}
\sum_{\substack{r_1 SF\\(r_1,d)=1}}\frac{\overline{g(\omega^a\lambda^bd,r_1)}}{N(r_1)^{s-z+5/3}}&=\prod_{\pi\mid d}\left(1-N(\pi)^{-3\bar{s}+3z-3}\right)^{-1}\sum_{c|d}\frac{\mu(c)g(\omega^a\lambda^bd/c,c)}{N(c)^{2\bar{s}-2z+7/3}}\nonumber\\
&\times\overline{\psi(\omega^a\lambda^bd/c,s-z+5/3)}.
\end{align}

\begin{lemma}\label{poles of integrand}
Fix $z\in [0,1]$. The integrand in \eqref{int of MT(z)} is analytic for $\Re(s)>z-\frac{2}{3}$ except for three poles at $s=0, s=z-1/3,$ and $s=z-2/3$. The poles are simple for $z\neq 1/3,2/3$, but two of them combine to give a double pole at $s=0$ when $z=1/3$ and $z=2/3.$
\end{lemma}

\begin{proof}
 Exchanging the $c$ and $d$ sums in \eqref{2nd application of Thm} and \eqref{exchanging r_1,d} by writing $d=d'c$, we get that
 \begin{equation*}
\mathcal{M}_2'(z)=Q^{4/3-z}\widetilde{w}\left(4/3-z\right)\frac{1}{2\pi i}\int_{(3)}B^s\mathcal{F}_z(s)\,ds,
\end{equation*}
where
 \begin{align}\label{mathfrak S}
\mathcal{F}_z(s)&:=\frac{1}{s}\frac{\pi^{2z+\frac{2}{3}} \Gamma(1-z)G(s)g_{1-z}(s)}{2^{\frac{7}{3}-2z}3^{z+3}\Gamma(2/3)\Gamma(z)\zeta_{\mathbb{Q}(\omega)}(2)}\sum_{k\geq 0}3^{k(-s-1+z)+s/2}\sum_{a,b=0}^2\delta_{a,b}\nonumber\\
&\times\sum_{c}\frac{\mu^2(c)}{N(c)^{2\bar{s}+s-3z+4}\sigma(c)}\sum_{d'}\frac{\mu(d')g(\omega^a\lambda^bd',c)\overline{g(\omega^a\lambda^b,cd')}}{N(d')^{s-z+\frac{5}{3}}\sigma(d')}\prod_{\pi\mid d'c}\left(1-N(\pi)^{-3\bar{s}+3z-3}\right)^{-1}\nonumber\\
&\times\overline{\psi(\omega^a\lambda^bd',s-z+5/3)}\sum_{\substack{r_3\\(r_3,cd')=1}}\frac{1}{N(r_3)^{3s-3z+3}}\prod_{\pi\mid r_3}\left(1+\frac{1}{N(\pi)} \right)^{-1}.
\end{align}
 We note that the coprimality condition $(r_3,cd')=1$ in the $r_3$ sum in \eqref{mathfrak S} will not affect the location of the possible poles coming from the sums over $c$ and $d'$. This follows from writing
 \begin{align*}
 &\sum_{\substack{r_3\\(r_3,cd')=1}}\frac{1}{N(r_3)^{3s-3z+3}}\prod_{\pi\mid r_3}\left(1+\frac{1}{N(\pi)} \right)^{-1}\\
 &=\prod_{\substack{\pi\equiv 1(\mathrm{mod}\ 3)\\\pi\nmid cd'}}\left(1+\left(1+\frac{1}{N(\pi)}\right)^{-1}\left(\frac{1}{N(\pi)^{3s-3z+3}}+\frac{1}{N(\pi^2)^{3s-3z+3}}+O\left(N(\pi)^{-6s+6z-6}\right)\right)\right)\\
 &=\mathcal{E}(s,z)\prod_{\substack{\pi\equiv 1(\mathrm{mod}\ 3)\\\pi|cd'}}\left(1+\frac{N(\pi)}{\left(N(\pi)+1\right)\left(N(\pi)^{3s-3z+3}-1\right)}\right)^{-1},
 \end{align*}
 where
\begin{align*}
\mathcal{E}(s,z)&:=\prod_{\pi\equiv 1(\mathrm{mod}\ 3)}\left(1+\frac{N(\pi)}{\left(N(\pi)+1\right)\left(N(\pi)^{3s-3z+3}-1\right)}\right)\\
&=\left(1-\frac{1}{3^{3s-3z+3}}\right)\zeta_{\mathbb{Q}(\omega)}(3s-3z+3)\prod_{\pi\equiv 1(\mathrm{mod}\ 3)}\left(1+O\left(N(\pi)^{-3s+3z-4}\right)\right)
\end{align*}
is analytic for $\Re(s)>z-1$ and has a pole at $s=z-2/3$. Since $(d',c)=1$, we can thus write
\begin{align*}
\mathcal{F}_z(s)&=\frac{1}{s}\frac{\pi^{2z+\frac{2}{3}} \Gamma(1-z)G(s)g_{1-z}(s)}{2^{\frac{7}{3}-2z}3^{z+3}\Gamma(2/3)\Gamma(z)\zeta_{\mathbb{Q}(\omega)}(2)}\mathcal{E}(s,z)\sum_{k\geq 0}3^{k(-s-1+z)+s/2}\sum_{a,b=0}^2\delta_{a,b}\nonumber\\
&\times\sum_{c}\frac{\mu^2(c)}{N(c)^{2\bar{s}+s-3z+4}\sigma(c)}\prod_{\pi|c}\left(1+\frac{N(\pi)}{\left(N(\pi)+1\right)\left(N(\pi)^{3s-3z+3}-1\right)}\right)^{-1}\nonumber\\
\end{align*}
\begin{align}\label{mathfrak S2}
&\times \prod_{\pi|c}\left(1-N(\pi)^{-3\bar{s}+3z-3}\right)^{-1} \sum_{d'}\frac{\mu(d')g(\omega^a\lambda^bd',c)\overline{g(\omega^a\lambda^b,cd')}}{N(d')^{s-z+\frac{5}{3}}\sigma(d')}\overline{\psi(\omega^a\lambda^bd',s-z+5/3)}\nonumber\\
&\times\prod_{\pi|d'}\left(\left(1-N(\pi)^{-3\bar{s}+3z-3}\right)\left(1+\frac{N(\pi)}{\left(N(\pi)+1\right)\left(N(\pi)^{3s-3z+3}-1\right)}\right)\right)^{-1}.
\end{align}
  Now from Lemma \ref{Patterson}, we know that $\overline{\psi(\omega^a\lambda^bd',s-z+5/3)}$ has a simple pole at $s=z-1/3$. There are also two other simple poles: $s=z-2/3$ that comes from $\mathcal{E}(s,z)$, and the $s=0$ pole. The last possible poles may come from the sums over $d'$ and $c$ in \eqref{mathfrak S2}. We claim that the sum over $d'$ is absolutely convergent for $\Re(s)>z-\frac{5}{6}+\epsilon$. To see this, we use the convexity bound of Lemma \ref{Patterson2} that works in the range $1+\epsilon\leq \Re(s)-z+\frac{5}{3}\leq \frac{4}{3}-\epsilon$, to bound the $d'$ sum in \eqref{mathfrak S2} as follows: bound $\psi$ by the convexity bound, bound $\sigma(d')$ by $\sigma(d')>N(d')$, and bound everything else trivially, to get
  \begin{align*}
  &\sum_{d'}\frac{\mu(d')g(\omega^a\lambda^bd',c)\overline{g(\omega^a\lambda^b,cd')}}{N(d')^{s-z+\frac{5}{3}}\sigma(d')}\prod_{\pi\mid d'}\left(1+\frac{N(\pi)}{\left(N(\pi)+1\right)\left(N(\pi)^{3s-3z+3}-1\right)}\right)^{-1}\\
  & \hspace{3.5cm}\times\prod_{\pi\mid d'}\left(1-N(\pi)^{-3\bar{s}+3z-3}\right)^{-1}\overline{\psi(\omega^a\lambda^bd',s-z+5/3)}\\
  &\ll \sum_{d'}\frac{1}{N(d')^{\frac{3}{2}(\Re(s)-z)+\frac{9}{4}+\epsilon}},
  \end{align*}
  for any $\epsilon>0$. Similarly, we can see that the sum over $c$ in \eqref{mathfrak S2} is absolutely convergent for $\Re(s)>z-1$. This completes the proof.
 \end{proof}
 
  In light of Lemma \ref{poles of integrand} we can now move the line of integration in (\ref{int of MT(z)}) depending on the range of $z$.

 \underline{If $\frac{1}{3}<z<\frac{2}{3}$}, we move the contour in (\ref{int of MT(z)}) to $s=z-\frac{2}{3}+\epsilon$ to cross the two simple poles at $s=z-1/3$ and $s=0$ only. The new line contributes $O(Q^{4/3-z}B^{z-2/3+\epsilon})$, while the poles give 
 $$D'_z\widetilde{w}\left(4/3-z\right)Q^{4/3-z}B^{z-1/3}+E'_z\widetilde{w}\left(4/3-z\right)Q^{4/3-z},$$
 for constants $D_z':=\lim_{s\to z-1/3}(s-z+1/3)\,\mathcal{F}_z(s)$, and
 \begin{align}\label{E'_z}
E'_z:=\lim_{s\to 0}s\,\mathcal{F}_z(s).
\end{align} 
  Using \eqref{2nd application of Thm} in \eqref{exchanging r_1,d} and then \eqref{exchanging r_1,d} in \eqref{int of  MT(z)}, we obtain
  \begin{align}\label{formula for D'_z}
  D_z'&=\frac{\pi^{2z+\frac{2}{3}} \Gamma(1-z)Q^{\frac{4}{3}-z}\widetilde{w}\left(\frac{4}{3}-z\right)}{2^{\frac{7}{3}-2z}3^{z+3}\Gamma(2/3)\Gamma(z)\zeta_{\mathbb{Q}(\omega)}(2)}\sum_{a,b=0}^2\delta_{a,b}\sum_{k\geq 0}3^{\frac{-2k}{3}+\frac{z}{2}-\frac{1}{6}}\frac{G(z-1/3)}{z-1/3}g_{1-z}(z-1/3)\nonumber\\
  &\times  \sum_{r_3}\frac{1}{N(r_3)^{2}}\prod_{\pi\mid r_3}\left(1+\frac{1}{N(\pi)} \right)^{-1}\sum_{\substack{(d,r_3)=1}}\frac{\mu(d)\overline{g(\omega^a\lambda^b,d)}}{N(d)^{4/3}\sigma(d)}\prod_{\pi\mid d}\left(1-N(\pi)^{-2}\right)^{-1}\nonumber\\
  &\times \sum_{c|d}\frac{\mu(c)g(\omega^a\lambda^bd/c,c)}{N(c)^{5/3}}\lim_{s\to z-\frac{1}{3}}\left(s-z+\frac{1}{3}\right)\overline{\psi(\omega^a\lambda^bd/c,s-z+5/3)}.
  \end{align}
  
\begin{lemma} For any $z\in [0,1]-\{\frac{1}{3}\}$, the constant $D_z'$ given in \eqref{formula for D'_z} can be written as
  \begin{align}\label{Cubic D'_z}
D'_z&=\frac{\pi^{2z+\frac{7}{3}}\Gamma(1-z)g_{1-z}(z-1/3)}{2^{\frac{11}{3}-2z}3^{\frac{z}{2}+\frac{14}{3}}\Gamma^2(2/3)\Gamma(z)\zeta_{\mathbb{Q}(\omega)}^2(2)}\sum_{k\geq 0}3^{-2k/3}\sum_{a,b=0}^2|\delta_{a,b}|^2\nonumber\\
&\times\frac{G(z-1/3)}{z-1/3}\prod_{\pi\equiv 1(\mathrm{mod}\ 3)}\left(\frac{N(\pi)(N(\pi)^2+N(\pi)-1)}{(N(\pi)-1)(N(\pi)+1)^2} \right).
\end{align}
\end{lemma}
\begin{proof}
 Using Lemma \ref{Patterson} once again, we have that
 \begin{equation}\label{residue formula}
\lim_{s\to z-\frac{1}{3}}\left(s-z+\frac{1}{3}\right)\overline{\psi(\omega^a\lambda^bd/c,s-z+5/3)}= \frac{(2\pi)^{5/3}\overline{\delta_{a,b}}g(\omega^a\lambda^b,d/c)}{3^{3/2}8\Gamma(2/3)\zeta_{\mathbb{Q}(\omega)}(2)}N(d/c)^{-2/3}.
 \end{equation}
The latter Gauss sum combines nicely with $g(\omega^a\lambda^bd/c,c)$ in \eqref{formula for D'_z} to give
 \[g(\omega^a\lambda^bd/c,c)g(\omega^a\lambda^b,d/c)=g(\omega^a\lambda^b,d),\]
 which is surprisingly just the conjugate of the remaining Gauss sum $g(\omega^a\lambda^b,d)$ in \eqref{formula for D'_z}, and so they multiply neatly to give their size $N(d)$. Now we collect the terms related to $d$ in \eqref{formula for D'_z} and write them as a product. We have
\begin{align}\label{d sum as product}
&\sum_{(d,r_3)=1}\frac{\mu(d)}{N(d)\sigma(d)}\prod_{\pi\mid d}\left(1-N(\pi)^{-2}\right)^{-1}\sum_{c|d}\frac{\mu(c)}{N(c)}\nonumber\\
&=\prod_{\pi\equiv 1(\mathrm{mod}\ 3)}\left(1-\frac{1}{(N(\pi)+1)^2}\right)\prod_{\pi|r_3}\left(1-\frac{1}{(N(\pi)+1)^2}\right)^{-1}.
\end{align}
The product over $\pi|r_3$ in \eqref{d sum as product} combined with the terms related to $r_3$ in \eqref{formula for D'_z} give
\begin{align*}
&\sum_{r_3}\frac{1}{N(r_3)^2}\prod_{\pi|r_3}\left(1+\frac{1}{N(\pi)}\right)^{-1}\prod_{\pi|r_3}\left(1-\frac{1}{(N(\pi)+1)^2}\right)^{-1}=\sum_{r_3}\frac{1}{N(r_3)^2}\prod_{\pi|r_3}\left(1-\frac{1}{N(\pi)+2}\right).
\end{align*}
Now write
\[\prod_{\pi|r_3}\left(1-\frac{1}{N(\pi)+2}\right)=\sum_{d|r_3}\frac{\mu(d)}{f(d)},\] where $f$ is a multiplicative function that satisfies $f(\pi)=N(\pi)+2$ for any prime $\pi\in \mathbb{Z}[\omega]$. Therefore, the sum over $r_3$ finally equals
\begin{align*}
\sum_{r_3}\frac{1*\frac{\mu}{f}(r_3)}{N(r_3)^2}&=\prod_{\pi\equiv 1(\mathrm{mod}\ 3)}\left(1-\frac{1}{N(\pi)^2}\right)^{-1}\prod_{\pi\equiv 1(\mathrm{mod}\ 3)}\left(1-\frac{1}{N(\pi)^2(N(\pi)+2)}\right)\\
&=\prod_{\pi\equiv 1(\mathrm{mod}\ 3)}\left(\frac{N(\pi)^2+N(\pi)-1}{(N(\pi)-1)(N(\pi)+2)}\right).
\end{align*}
Combining this product with the remaining one in \eqref{d sum as product}, we obtain that the full product that results from all this computation is
\[\prod_{\pi\equiv 1(\mathrm{mod}\ 3)}\left(\frac{N(\pi)(N(\pi)^2+N(\pi)-1)}{(N(\pi)-1)(N(\pi)+1)^2} \right).\]
Using this, and collecting the remaining factors in \eqref{residue formula}, and then replacing in \eqref{formula for D'_z}, we obtain the desired form \eqref{Cubic D'_z} in the Lemma. 
\end{proof}

 \underline{If $0\leq z<\frac{1}{3}$}, we now have that the two poles satisfy $z-1/3<0$. In this case, when moving the contour in (\ref{int of MT(z)}) to $s=z-\frac{2}{3}+\epsilon$, the main and secondary terms will switch. That is, we get the asymptotic
 $$E'_z\widetilde{w}\left(4/3-z\right)Q^{4/3-z}+D'_z\widetilde{w}\left(4/3-z\right)Q^{4/3-z}B^{z-1/3}+O(Q^{4/3-z}B^{z-2/3+\epsilon}).$$

  \underline{If $\frac{2}{3}\leq z\leq 1$}, we can only capture the $s=z-1/3$ pole since $0<z-2/3<z-1/3$. We again move the contour to $s=z-\frac{2}{3}+\epsilon$ to get 
  $$D'_z\widetilde{\omega}\left(4/3-z\right)Q^{4/3-z}B^{z-1/3}+O(Q^{4/3-z}B^{z-2/3+\epsilon}).$$
  
 Finally, \underline{at $z=1/3$}, we move the line of integration in (\ref{int of MT(z)}) to $s=-1/3+\epsilon$ to cross a double pole at $s=0$ now. The new line contributes $O(QB^{-1/3+\epsilon})$, while the double pole gives
 \[C_{1/3}'\widetilde{w}(1)Q\log B+C_2'\widetilde{w}(1)Q \]
 which comes from the evaluation
\[Q\widetilde{w}(1)\lim_{s\to 0}s^2\frac{d}{ds}\left(\mathcal{F}_{1/3}(s)B^s\right)=Q\widetilde{w}(1)\left(\lim_{s\to 0}s^2\mathcal{F}_{1/3}(s)\log B+\lim_{s\to 0}s^2\mathcal{F}'_{1/3}(s)\right). \]
This implies that $C_{1/3}'=\lim_{s\to 0}s^2\mathcal{F}_{1/3}(s)\log B$ and
\begin{align}\label{Cubic C_2'}
C'_2&=\lim_{s\to 0}s^2\mathcal{F'}_{1/3}(s).
\end{align}
The computation of the constant $C_{1/3}'$ follows similarly as that of $D_z'$. In fact, the latter computation remains valid but with $z$ replaced by $1/3$. In addition using the reflection formula $\Gamma(2/3)\Gamma(1/3)=\frac{2\pi}{\sqrt{3}}$, we get that the constant equals
\begin{align}\label{Cubic C'_1/3}
C_{1/3}'&=\frac{\pi^{2}}{2^{4}3^{\frac{13}{3}}\zeta_{\mathbb{Q}(\omega)}^2(2)}\sum_{k\geq 0}3^{-2k/3}\sum_{a,b=0}^2|\delta_{a,b}|^2\prod_{\pi\equiv 1(\mathrm{mod}\ 3)}\left(\frac{N(\pi)(N(\pi)^2+N(\pi)-1)}{(N(\pi)-1)(N(\pi)+1)^2} \right).
\end{align}
Finally, we collect everything we proved above in the following theorem.

\begin{theorem}\label{M_2'} Let $\epsilon>0$ and $D'_z,E'_z,C_{1/3}'$ and $C_2'$ be the constants given for $z\in [0,1]$ in \eqref{Cubic D'_z}, \eqref{E'_z}, \eqref{Cubic C'_1/3} and \eqref{Cubic C_2'} respectively. Then the following are true.

If $z\in(1/3,2/3)$, we have
 \[\mathcal{M}_2'(z)=D'_z\widetilde{w}\left(4/3-z\right)Q^{4/3-z}B^{z-1/3}+E'_z\widetilde{w}\left(4/3-z\right)Q^{4/3-z}+O(Q^{4/3-z}B^{z-2/3+\epsilon}). \]

If $z\in[2/3,1]$,
 \[\mathcal{M}_2'(z)=D'_z\widetilde{w}\left(4/3-z\right)Q^{4/3-z}B^{z-1/3}+O(Q^{4/3-z}B^{z-2/3+\epsilon}). \]
 
 If $z\in[0,1/3)$,
 \[\mathcal{M}_2'(z)=E'_z\widetilde{w}\left(4/3-z\right)Q^{4/3-z}+D'_z\widetilde{w}\left(4/3-z\right)Q^{4/3-z}B^{z-1/3}+O(Q^{4/3-z}B^{z-2/3+\epsilon}). \]

  At $z=1/3$,
\[\mathcal{M}_2'(1/3)=C_{1/3}'\widetilde{\omega}(1)Q\log B+C_2'\widetilde{\omega}(1)Q+O(QB^{-1/3+\epsilon}). \]
\end{theorem}

It remains to show how these constants are related to the ones from the principal sum. This is done in the next theorem.
\begin{theorem}\label{Constants}
Let $C_{1/3},C_{1/3}',D_z$ and $D_z'$ be the constants given by \eqref{Cubic C_1/3}, \eqref{Cubic C'_1/3}, \eqref{Cubic D_z}, and \eqref{Cubic D'_z} respectively  for any $z\in[0,1]$. Then $C_{1/3}'=C_{1/3}$ and $D_z'=-D_z$.
\end{theorem}
\begin{proof}
The sum over $k$ and the sum over $a,b$ in each of \eqref{Cubic C'_1/3} and \eqref{Cubic D'_z} combine rather nicely to give
\[\sum_{k\geq 0}3^{-2k/3}\sum_{a,b=0}^2|\delta_{a,b}|^2=\left(\frac{1}{1-\frac{1}{\sqrt[3]{3^2}}}\right)\left(1+3^{1/3}\right)=\frac{3^{1/3}}{1-\frac{1}{\sqrt[3]{3}}}.\]
We also have
\begin{align*}
&\frac{1}{\zeta_{\mathbb{Q}(\omega)}(2)}\prod_{\pi\equiv 1(\mathrm{mod}\ 3)}\left(\frac{N(\pi)(N(\pi)^2+N(\pi)-1)}{(N(\pi)-1)(N(\pi)+1)^2} \right)\\
&=\frac{8}{9}\prod_{\pi\equiv 1(\mathrm{mod}\ 3)}\left(1-\frac{1}{N(\pi)^2}\right)\prod_{\pi\equiv 1(\mathrm{mod}\ 3)}\left(\frac{N(\pi)(N(\pi)^2+N(\pi)-1)}{(N(\pi)-1)(N(\pi)+1)^2} \right)\\
&=\frac{8}{9}\prod_{\pi\equiv 1(\mathrm{mod}\ 3)}\left(1-\frac{1}{N(\pi)(N(\pi)+1)}\right).
\end{align*}
Using these in \eqref{Cubic C'_1/3}, we arrive at
\begin{align*}
C_{1/3}'&=\left(\frac{1}{1-\frac{1}{\sqrt[3]{3}}}\right)\frac{\pi^2}{23^6\zeta_{\mathbb{Q}(\omega)}(2)}\prod_{\pi\equiv 1(\mathrm{mod}\ 3)}\left(1-\frac{1}{N(\pi)(N(\pi)+1)}\right),
\end{align*}
and thus $C_{1/3}=C_{1/3}'$. Using additionally the fact that
\[ g_{1-z}(z-1/3)=(2\pi)^{1/3-z}\frac{\Gamma(2/3)}{\Gamma(1-z)}\] in \eqref{Cubic D'_z}, together with the reflection formula
$\Gamma(2/3)\Gamma(1/3)=\frac{2\pi}{\sqrt{3}},$
we conclude that
\begin{equation*}
D'_z=\left(\frac{1}{1-\frac{1}{\sqrt[3]{3}}}\right)\frac{2^{z-\frac{4}{3}}\pi^{z+\frac{5}{3}}\Gamma(1/3)}{3^{\frac{z}{2}+\frac{35}{6}}\Gamma(z)\zeta_{\mathbb{Q}(\omega)}(2)}\frac{G(z-1/3)}{z-1/3}\prod_{\pi\equiv 1(\mathrm{mod}\ 3)}\left(1-\frac{1}{N(\pi)(N(\pi)+1)}\right)
\end{equation*}
and indeed get $D'_z=-D_z$ since $G$ is an even function from Proposition \ref{AFE}.
\end{proof}

Combining Theorems \ref{Constants} and \ref{M_2'} with Corollary \ref{M_2''}, the proof of Theorem \ref{Asymptotics} is now complete.

\section{The Quadratic Case Revisited}

Let $q\in\mathbb{Z}_{\geq 1}$ be odd, SF and $q\equiv 1\pmod{4}$. Let $\chi_q$ be the quadratic character given by the Kronecker symbol $\chi_q(a)=\left(\frac{a}{q}\right)_2$. Note that $\chi_q$ is primitive and has conductor $q$ in this case, and the quadratic $L$-function $L(z,\chi_q)$ satisfies the functional equation
\[\Lambda(z,\chi_q)=\left(\frac{q}{\pi}\right)^{z/2}\Gamma\left(\frac{z}{2}\right)L(z,\chi_q)=\Lambda(1-z,\chi_q). \]
Let $Q>2$ and consider the family \[\cF_2=\{\chi_q\;:\;q\in\mathbb{Z}_{>0},\; q\equiv 1\pmod 4,\; \text{and}\; \mu^2(q)=1 \}\]
 of all such characters having conductor $q\leq Q$.

As in the cubic case, the value of the $L$-function inside the critical strip is given by the approximate functional equation. 
\begin{proposition}\label{AFE0} Let $G(u)$ be any function which is holomorphic and bounded in the strip $-4<\Re(u)<4$, even, and normalized by $G(0)=1$. Let $X>0$, then for $z=\sigma+it$ in the strip $0\leq \sigma\leq 1$ we have
\[L(z,\chi_q)=\sum_{n\geq 1}\frac{\chi_q(n)}{n^z}V_z\left(\frac{n}{\sqrt{q}}\right)+\epsilon(z,\chi_q)\sum_{n\geq 1}\frac{\chi_q(n)}{n^{1-z}}V_{1-z}\left(\frac{n}{\sqrt{q}}\right), \]
where \[V_z(y)=\frac{1}{2\pi i}\int_{(3)}y^{-u}\frac{G(u)}{u}g_z(u)\,du \ \
\text{with} \ \ g_z(u)=\pi^{-u/2}\frac{\Gamma\left(\frac{z+u}{2}\right)}{\Gamma\left(\frac{z}{2}\right)}, \]
and \[\epsilon(z,\chi_q)=\left(\frac{q}{\pi}\right)^{1/2-z}\frac{\Gamma\left(\frac{1-z}{2}\right)}{\Gamma\left(\frac{z}{2}\right)}. \]
\end{proposition}
We will also use the following bounds on the $V_z$ function and its derivatives. These can be again deduced from \cite[Proposition 5.4]{IwK}.

\begin{lemma}\label{Bounds on V_s0}
 For $\Re(z+1)\geq 3\alpha>0,$ the derivatives of $V_z(y)$ satisfy
 \[y^aV_z^{(a)}(y)\ll_{a,E}\left(1+\frac{y}{\sqrt{(|z|+3)}}\right)^{-E} \]
 for any $E>0$, and
 \[ y^aV_z^{(a)}(y)=\delta_a+O_{a,\alpha}\left(\left(\frac{y}{\sqrt{|z|+3}}\right)^{\alpha}\right)\]
 where $\delta_0=1$ and $\delta_a=0$ for $a\geq 1$.
 \end{lemma}

From Proposition \ref{AFE0}, the first moment over the family $\mathcal{F}_2$ equals
\[\sum_{\substack{q\geq 1\, SF\\q\equiv 1(\mathrm{mod}\ 4) }}L(z,\chi_q)\omega\left(q/Q \right)=\mathcal{M}_1(z)+\mathcal{M}_2(z), \]
where $\mathcal{M}_1$ is the principal sum given by
\[\mathcal{M}_1(z)=\sum_{n\geq 1}\frac{1}{n^z}\sum_{\substack{q\geq 1\\q\, SF\\ q\equiv 1(\mathrm{mod}\ 4)}}\chi_q(n)V_z\left(\frac{n}{\sqrt{q}}\right)\omega\left(q/Q \right) \]
and $\mathcal{M}_2$ is the dual sum given by
\[\mathcal{M}_2(z)=\pi^{z-1/2}\frac{\Gamma\left(\frac{1-z}{2}\right)}{\Gamma\left(\frac{z}{2}\right)}\sum_{n\geq 1}\frac{1}{n^{1-z}}\sum_{\substack{q\geq 1\\q\, SF\\ q\equiv 1(\mathrm{mod}\ 4)}}\frac{\chi_q(n)}{q^{z-1/2}}V_{1-z}\left(\frac{n}{\sqrt{q}}\right)\omega\left(q/Q \right). \]
Note that here we use a balanced approximate functional equation (unlike in the cubic case) since both sums are of the same size, and at $z=1/2$ the principal sum is in fact equal to the dual sum.

In the rest of this section, we prove the following asymptotics.

\begin{theorem}\label{Asymptotics0}
Let $\epsilon>0$. For any $z\in(0,1)-\{\frac{1}{2}\}$, we have 
\[\mathcal{M}_1(z)=C_{z}\widetilde{w}(1)Q+D_z \widetilde{w}(5/4-z/2)Q^{5/4-z/2}+O(Q^{1-\frac{z}{2}+\epsilon})\]
and 
\[\mathcal{M}_2(z)=-D_z \widetilde{w}(5/4-z/2)Q^{5/4-z/2}+C_z'\widetilde{w}(3/2-z)Q^{3/2-z}+O(Q^{1-\frac{z}{2}+\epsilon}),\]
for constants $C_z,C'_z$ and $D_z$ given explicitly in \eqref{Quad C_z}, \eqref{Quad C_z'}, and \eqref{Quad D_z} respectively.

At $z=0$ we have
\[\mathcal{M}_1(0)=D_0\widetilde{w}(5/4)Q^{5/4}+O(Q^{1+\epsilon})\]
and
\[\mathcal{M}_2(0)=-D_0\widetilde{w}(5/4)Q^{5/4}+C'_0\widetilde{w}(3/2)Q^{3/2}+O(Q^{1+\epsilon}).\]

At $z=1$ we have
\[\mathcal{M}_1(1)=C_{1}\widetilde{w}(1)Q+D_1 \widetilde{w}(3/4)Q^{3/4}+O(Q^{1/2+\epsilon}),\]
and
\[\mathcal{M}_2(1)=-D_1\widetilde{w}(3/4)Q^{3/4}+O(Q^{1/2+\epsilon}).\]

At $z=1/2$, we have 
\[\mathcal{M}_1(1/2)=\mathcal{M}_2(1/2)=C_{1/2}\widetilde{w}(1)Q\log Q+C_{2}Q+O(Q^{3/4+\epsilon}),\]
where $C_{1/2}$ and $C_2$ are given explicitly in \eqref{Quad C_1/2} and \eqref{Quad C_2} respectively.
\end{theorem}

As in the cubic case, we again here notice a cancellation between two of the terms in the asymptotic for $z\neq 1/2$, namely the $Q^{\frac{5}{4}-\frac{z}{2}}$ term. For $z> 1/2,$ the secondary term of $\mathcal{M}_1(z)$ cancels with the main term of $\mathcal{M}_2(z)$. And for $z<1/2,$ the main term of $\mathcal{M}_1(z)$ cancels with the secondary term of $\mathcal{M}_2(z)$. We mention that an explicit cancellation was also observed in the methods of \cite{Young} and \cite{Florea} for $z=1/2$. 

\subsection{The Principal Sum}

The computation here is similar to the cubic case. We first sieve over squarefree to get
 \begin{align*}
\mathcal{M}_1(z)&=\sum_{n\geq 1}\frac{1}{n^z}\sum_{\ell\geq 1}\mu(\ell)\left(\frac{n}{\ell^2}\right)_2\sum_{\substack{q\geq 1\\ q\equiv 1(\mathrm{mod}\ 4)}}\left(\frac{n}{q}\right)_2V_z\left(\frac{n}{\sqrt{\ell^2q}}\right)\omega\left(\frac{\ell^2q}{Q}\right)\\
&=\sum_{n\geq 1}\frac{1}{n^z}\sum_{\substack{\ell\geq 1\\(\ell,n)=1}}\mu(\ell)\mathcal{M}_1(\ell,n,z),
\end{align*}
where 
\[\mathcal{M}_1(\ell,n,z)=\sum_{\substack{q\geq 1\\ q\equiv 1(\mathrm{mod}\ 4)}}V_z\left(\frac{n}{\ell\sqrt{q}}\right)\omega\left(\frac{\ell^2q}{Q}\right).\]
Using Mellin inversion, we can write
\[V_{z}\left(\frac{n}{\ell\sqrt{q}}\right)\omega\left(\frac{\ell^2q}{Q}\right)=\frac{1}{2\pi i}\int_{(2)} \left(\frac{Q}{\ell^2q} \right)^s\widetilde{f}_z(s,n)\,ds\]
where $$\widetilde{f}_z(s,n)=\int_{0}^{\infty}V_{z}\left(\frac{n}{\sqrt{Qx}}\right)\omega(x)x^{s-1}\,dx$$
again satisfies
\begin{equation}\label{f_z}
\widetilde{f}_z(s,n)\ll(1+\lvert s\rvert)^{-E}\left(1+\frac{n}{\sqrt{Q}}\right)^{-E},
\end{equation}
for $\Re(s)\geq \frac{1}{4}$ and any $E>0$ and $0\leq z\leq 1$. With this notation, we have
\[ \mathcal{M}_1(\ell,n,z)=\frac{1}{2\pi i}\int_{(2)} \sum_{\substack{q\geq 1\\ q\equiv 1(\mathrm{mod}\ 4)}}\left(\frac{n}{q}\right)_2\left(\frac{Q}{\ell^2q} \right)^s\widetilde{f}_z(s,n)\,ds. \]
As $q\equiv 1\mod 4$, quadratic reciprocity gives $\chi_q(n)=\chi_n(q)$, and this yields  
\begin{align*}
\mathcal{M}_1(\ell,n,z)&=\frac{1}{2\pi i}\int_{(2)}\left(\frac{Q}{\ell^2}\right)^s\sum_{\substack{q\geq 1\\ q\equiv 1(\mathrm{mod}\ 4)}}\frac{\left(\frac{q}{n}\right)_2}{q^s}\widetilde{f}_z(s,n)\,ds\\
&=\frac{1}{2}\frac{1}{2\pi i}\int_{(2)}\left(\frac{Q}{\ell^2}\right)^s\sum_{\psi \mod 4}L(s,\chi_n\psi)\widetilde{f}_z(s,n)\,ds,
\end{align*}
where the last equality follows from the orthogonality relation
\[\sum_{\psi \mod 4}\psi(q)=\begin{cases}
2& \text{if} \ \ q\equiv 1\pmod 4 \\
0& \text{if} \ \ q\not\equiv 1\pmod 4,
\end{cases} 
\]
and this sum is over all Dirichlet characters mod $4$. Now each $L$-function $L(s,\psi\chi_n)$ has a pole at $s=1$ only when $n$ is a square and $\psi$ is trivial. So we estimate $\mathcal{M}_1(z)$ by moving the contour to the half line. We denote by $\mathcal{M}_0(z)$ the contribution of $\mathcal{M}_1(z)$ at these poles, and by $\mathcal{M}'_0(z)$ the contribution of the half line.

In order to bound the error term $\mathcal{M}_0'(z)$, we will need the following upper bound on the second moment of the quadratic $L$-functions at the point $1/2+it$.

\begin{theorem}\label{Bound on second moment} For any $X>1$ and $t\in \mathbb{R}$ we have the bound
\[\sum_{n\leq X}\lvert L(1/2+it,\chi_n)\rvert^2 \ll_{\epsilon} X^{1+\epsilon}\left(1+|t|\right)^{1/2+\epsilon}\]
for any $\epsilon>0$. 
\end{theorem}

This result may be found in the work of Soundarararajan \cite[Lemma 2.5]{Sound2000}, and is a consequence of the following quadratic large sieve proven by Heath-Brown \cite{HB1995}.

\begin{lemma}\label{quadrativ large sieve}
Let $M,N>0$ and $c_n$ be an arbitrary sequence of complex numbers where $n$ runs in $\mathbb{Z}_{\geq 1}$. Then
 \[\psum_{\substack{m\leq M}}\left\lvert\, \psum_{n\leq N}c_n\left(\frac{n}{m}\right)_2\right\rvert^2\ll_{\epsilon} \left(M+N\right)\left(MN \right)^{\epsilon}\psum_{n\leq N}\lvert c_n \rvert^2\]
for any $\epsilon>0$, where the star on the sum indicates that we are summing over positive odd squarefree values.
\end{lemma}

Therefore we can now bound $\mathcal{M}_0'(z)$ as follows
\begin{align*}
\mathcal{M}_0'(z)&\ll Q^{1/2}\sum_{\substack{n\geq 1\\ n\leq Q^{1/2+\epsilon}}}\frac{1}{n^z}\sum_{\substack{\ell\geq 1\\ \ell\leq Q^{1/2+\epsilon}}}\frac{1}{\ell}\int_{-\infty}^{\infty}\sum_{\psi \mod 4}\lvert L(1/2+it,\chi_n\psi)\rvert \lvert \tilde{f}_z(1/2+it,n) \rvert \,dt\\
&\ll_{E,\epsilon} Q^{1/2+E/2+\epsilon}\sum_{\psi \mod 4}\int_{-\infty}^{\infty}\sum_{n\leq Q^{1/2+\epsilon}}\frac{1}{n^{z+E}}\lvert L(1/2+it,\chi_n\psi)\rvert (1+|1/2+it|)^{-E}  \,dt,
\end{align*}
for any $E>0$ where we used the bound on $\tilde{f}_z(1/2+it,n)$ in \eqref{f_z}. Now we use Cauchy-Schwarz to split the sum over $n$ and then use the bound on the second moment in Theorem \ref{Bound on second moment} to finally get that for any choice of $E>0$,
\[\mathcal{M}_0'(z)\ll_{\epsilon} Q^{1-z/2+\epsilon}. \]

 When $\psi=\psi_0$ is the trivial character and $n=\square$, we write $n^2$ in place of $n$ and get that the contribution of the poles is
 \[\mathcal{M}_0 (z)=Q\sum_{n\geq 1}\frac{1}{n^{2z}}\sum_{\substack{\ell\geq 1\\(\ell,n)=1}}\frac{\mu(\ell)}{\ell^2}\widetilde{f}_z(1,n^2)Res_{s=1}L(s,\psi_0\chi_{n^2}).\]
Similar to the cubic case, we write
\[L(s,\psi_0\chi_{n^2})=\zeta(s)\prod_{p|2n}\left(1-\frac{1}{p^s}\right),\]
and 
\[\sum_{\substack{\ell\geq 1\\(\ell,n)=1}}\frac{\mu(\ell)}{\ell^2}=\frac{1}{\zeta(2)}\prod_{p|2n}\left(1-\frac{1}{p^2}\right)^{-1}. \]
Introducing the integral form of $V_{z}$ from Proposition \ref{AFE0} in $\widetilde{f}_z$, and then using the Mellin convolution formula for $\omega(x)$, we get that
\[\widetilde{f}_z(1,n^2)=\int_{0}^{\infty}V_{z}\left(\frac{n^2}{\sqrt{Qx}}\right)\omega(x)\,dx=\frac{1}{2\pi i}\int_{(3)}\left(\frac{\sqrt{Q}}{n^2}\right)^s\widetilde{\omega}(s/2+1)\frac{G(s)}{s}g_{z}(s)\,ds. \] 
Finally denote by \[Z(s):=\sum_{n\geq 1}\frac{1}{n^s}\left(1+\frac{1}{p}\right)^{-1}. \]
Therefore we obtain, using all of the above, and the fact that $\lim_{s\to 1}(s-1)\zeta(s)=1$,
\begin{equation}\label{Z(s)0}
\mathcal{M}_0(z)=\frac{Q}{3\zeta(2)}\frac{1}{2\pi i}\int_{(3)}Q^{s/2}Z(2s+2z)\widetilde{\omega}(1+s/2)\frac{G(s)}{s}g_{z}(s)\,ds.\
\end{equation}

Note that $Z(s)$ is holomorphic and bounded for $\Re(s)\geq 1+\delta>1$. And similar to the cubic case, we can write
\[Z(s)=\zeta(s)\prod_{p}\left(1-\frac{1}{p^s(p+1)}\right),\]
which will again give
\begin{align*}
Z(s)&=\frac{\zeta(s)}{\zeta(s+1)}\prod_p\left(1+O\left(\frac{1}{p^{s+2}}\right)\right).
\end{align*}
The Euler product on the right hand side is analytic for $\Re(s)>-1$, and the quotient of zeta functions has poles at $s=1$ and in the region where $\Re(s)<0$. Therefore, $Z(2s+2z)$ has poles at $s=1/2-z$ and $\Re(s)<-z$. 

Thus, for $1/2<z\leq 1$, we can move the contour of integration in (\ref{Z(s)0}) to $s=-z$ to cross two simple poles at $s=0$ and $s=1/2-z< 0$ only. The new contour contributes $O(Q^{1-z/2})$, while the two poles give $C_z\widetilde{\omega}(1)Q$ at $s=0$, which is the main term in this case; and $D_z\widetilde{\omega}(5/4-z/2)Q^{5/4-z/2}$ at $s=1/2-z$, which is the secondary term. The two constants are given as
\begin{equation}\label{Quad C_z}
C_z=\frac{\zeta(2z)}{3\zeta(2)}\prod_{p}\left(1-\frac{1}{p^{2z}(p+1)}\right),
\end{equation}
and 
\begin{equation}\label{Quad D_z}
D_z=\frac{\pi^{z/2-1/4}\Gamma(1/4)}{6\Gamma(z/2)\zeta(2)}\frac{G(1/2-z)}{1/2-z}\prod_{p}\left(1-\frac{1}{p(p+1)}\right),
\end{equation}
where we have used $$g_{z}(1/2-z)=\pi^{z/2-1/4}\frac{\Gamma(1/4)}{\Gamma(z/2)}.$$

For $0<z<1/2$, the same analysis follows but now the simple pole at $s=1/2-z$ is positive. Hence, we get the same terms with the same constants in the asymptotic but with the main and secondary terms switched. 

At $z=0$, we only move the line of integration to $s=\epsilon$ since there might be poles for $\Re(s)<0$. This gives an error term of $O(Q^{1+\epsilon})$ and a main term of $D_0\widetilde{\omega}(5/4)Q^{5/4}$ which comes from the pole at $s=1/2$. If we in addition assume GRH, we can now move the line to $s=-1/4+\epsilon$ to get the better error term $O(Q^{7/8 +\epsilon})$ and capture the second pole at $s=0$ which will contribute the secondary term $C_0\widetilde{\omega}(1)Q$.

However, at $z=1/2$, when moving the contour of integration in (\ref{Z(s)}) to $s=-1/2$, we now cross a double pole at $s=0$ only. The new contour contributes $O(Q^{3/4})$, while the double pole at $s=0$ now gives
\[C_{1/2}\widetilde{\omega}(1)Q\log Q+C_2Q,\]
where the constants are given by
\begin{equation}\label{Quad C_1/2}
C_{1/2}=\frac{1}{12\zeta(2)}\prod_{p}\left(1-\frac{1}{p(p+1)}\right),\\
\end{equation}
and 
\begin{align}\label{Quad C_2}
C_2&=\Bigg[\frac{\widetilde{\omega}'(1)}{12\zeta(2)}+\frac{\widetilde{\omega}(1)}{6\zeta(2)}G'(0)+\frac{\widetilde{\omega}(1)}{12\zeta(2)}\left(\frac{\Gamma'(1/4)}{\Gamma(1/4)}-\log(\pi)\right)\nonumber\\
&+\frac{\widetilde{\omega}(1)}{3\zeta(2)}\left(\gamma+\sum_{p}\frac{\log p}{p(p+1)-1}\right)\Bigg]\prod_p\left(1-\frac{1}{p(p+1)}\right).
\end{align}

\subsection{The Dual Sum}
In the case of quadratic characters, the computation of the dual sum is similar to that of the principal sum. Thus, proceeding similarly as we did in the principal sum case, we arrive at

\begin{align*}
\mathcal{M}_2(z)&=\frac{\pi^{z-1/2}\Gamma\left(\frac{1-z}{2}\right)}{\Gamma\left(\frac{z}{2}\right)}\sum_{n\geq 1}\frac{1}{n^{1-z}}\sum_{\substack{\ell\geq 1\\(\ell,n)=1}}\frac{\mu(\ell)}{\ell^{2z-1}}\\
& \times \frac{1}{2\pi i}\int_{(2)}\frac{Q^s}{\ell^{2s}}\sum_{\psi \mod 4}L(s+z-1/2,\psi\chi_n)\widetilde{f}_z(s,n)\,ds.
\end{align*}
The difference this time is that the pole and the contour will depend on the value of $z$. The $L$-function here has a pole at $s=3/2-z$ only when $n$ is a square and $\psi=\psi_0$ is the principal character. Therefore, for any $0\leq z\leq 1$, we move the line of integration to the line $s=1-z\geq 0$. Denote by $\mathcal{M}_2'(z)$ the contribution of the poles at $s=3/2-z$ and by $\mathcal{M}_2''(z)$ the contribution of the line at $s=1-z$. 

Similarly as we did above in bounding the error term, we have 
\[\mathcal{M}_2''(z)\ll Q^{1-z+E/2+\epsilon}\int_{-\infty}^{\infty}\sum_{\psi \mod 4}\sum_{n\leq Q^{1/2+\epsilon}}\frac{1}{n^{1-z+E}}\lvert L(1/2+it,\chi_n\psi)\rvert (1+|1-z+it|)^{-E}\,dt, \]
and splitting the sum over $n$ using Cauchy-Schwarz and then using the bound in Theorem \ref{Bound on second moment}, we get
\[\mathcal{M}_2''(z)\ll_{\epsilon}Q^{1-z/2+\epsilon}. \]
The poles when $n=\square$ and $\psi$ is principal give
\[\mathcal{M}_2(z)=Q^{3/2-z}\frac{\pi^{z-1/2}\Gamma\left(\frac{1-z}{2}\right)}{\Gamma\left(\frac{z}{2}\right)}\sum_{n\geq 1}\frac{1}{n^{2-2z}}\sum_{\substack{\ell\geq 1\\(\ell,n)=1}}\frac{\mu(\ell)}{\ell^{2}}\widetilde{f}_z(3/2-z,n^2)Res_{s=1}L(s,\psi_0\chi_{n^2}). \]
And so, we get as before,
\[\mathcal{M}_2(z)=Q^{3/2-z}\frac{\pi^{z-1/2}\Gamma\left(\frac{1-z}{2}\right)}{3\zeta(2)\Gamma\left(\frac{z}{2}\right)}\int_{(3)}Q^{s/2}Z(2s-2z+2)\frac{G(s)}{s}g_{1-z}(s)\widetilde{\omega}(s/2+3/2-z)\,ds. \]
Now, $Z(2s-2z+2)$ has poles at $s=z-1/2$ and for $\Re(s)<z-1$. Thus, for $1/2<z< 1$, we can move the contour of integration in (\ref{Z(s)}) to $s=z-1$ to cross two simple poles at $s=0$ and $s=z-1/2>0$ only. The new contour contributes $O(Q^{1-z/2})$, while the two poles give $D_z'\widetilde{\omega}(5/4-z/2)Q^{5/4-z/2}$ at $s=z-1/2$, which is the main term in this case; and $C_z'\widetilde{\omega}(3/2-z)Q^{3/2-z}$ at $s=0$, which is the secondary term here.  

We have that
\begin{equation}\label{Quad C_z'}
C_z'=\frac{\pi^{z-1/2}\Gamma(\frac{1-z}{2})}{3\zeta(2)\Gamma(z/2)}\zeta(2-2z)\prod_p\left(1-\frac{p^{2z}}{p^3+p^2}\right),
\end{equation}
 and using the fact that $G$ is an even function together with the symmetry relation
\[\frac{\pi^{z-1/2}\Gamma\left(\frac{1-z}{2}\right)}{\Gamma\left(\frac{z}{2}\right)}g_{1-z}(z-1/2)=g_{z}(1/2-z),\]
we get that $D_z=-D_z'$ (and so the $Q^{5/4-z/2}$ term cancels with the one from the principal sum).

For $0\leq z<1/2$, the same analysis follows but now the simple pole at $s=z-1/2$ is negative. Hence, we get the same terms with the same constants in the asymptotic but with the main term and secondary term switched in this case.

At $z=1$, we only move the line of integration to $s=\epsilon$ since there might be poles for $\Re(s)<0$. This gives an error term of $O(Q^{1/2+\epsilon})$ and a main term of $D_1'\widetilde{\omega}(3/4)Q^{3/4}$. If we in addition assume GRH, we can now move the line to $s=-1/4+\epsilon$ to get the better error term $O(Q^{3/8 +\epsilon})$ and capture the second pole at $s=0$ which will contribute the secondary term $C_1'\widetilde{\omega}(1/2)Q^{1/2}$.

 For $z=1/2$ recall that $\mathcal{M}_2(1/2)=\mathcal{M}_1(1/2)$. Therefore, we have finally proved Theorem \ref{Asymptotics0}. And Theorem \ref{Quadratic} will now follow by setting 
 \begin{equation}\label{Quad C}
   C:=2C_{1/2}
   \end{equation}
    where $C_{1/2}$ is given in \eqref{Quad C_1/2}, and 
\begin{equation}\label{Quad D}
 D:=2C_2
 \end{equation}
 where $C_2$ is given in \eqref{Quad C_2}.
 \begin{remark}\label{q= 3mod 4}
If we were working instead with the family where $q\equiv 3\pmod 4$, the only differences would be when we used quadratic reciprocity and the orthogonality relation. In this case we would get a factor of $-1$ from reciprocity, and the character $\psi'(q)=\psi(3q)$ from orthogonality. At the end this results in the same asymptotics but with the new constants equal to $-1$ times the previous ones.
 \end{remark}

\printbibliography

\end{document}